\documentclass[11pt, a4paper]{article}

\usepackage [utf8] {inputenc}		\usepackage[T1]{fontenc}
\usepackage{fullpage,amsmath,amsthm,amsfonts,graphicx, amscd}
\usepackage{float} 
\usepackage{tikz} 
\usetikzlibrary{decorations.pathreplacing}

\usepackage{anyfontsize} \usepackage{libertine} \usepackage[libertine]{newtxmath} \usepackage{eucal}

\newtheorem{theoreme}{Theorem}
\newtheorem{lemma}[theoreme]{Lemma}	
\newtheorem{proposition}[theoreme]{Proposition} 
\newtheorem{corollary}[theoreme]{Corollary}

\newtheorem{remark}[theoreme]{Remark}  
\newtheorem{definition}[theoreme]{Definition}
  
\newtheorem{notations}{Notations}

\DeclareMathOperator{\re}{Re}	\DeclareMathOperator{\im}{Im}

\newcommand{\norme}[1]{{\left\vert\kern-0.2ex\left\vert\kern-0.2ex\left\vert #1 
	\right\vert\kern-0.2ex\right\vert\kern-0.2ex\right\vert}}
\newcommand{\scal}[2]{\left\langle  {#1} , {#2}  \right\rangle} 

\DeclareMathOperator{\attr}{\mathbf A}
\DeclareMathOperator{\born}{\mathbf B}

\DeclareMathOperator{\Var}{Var}
\DeclareMathOperator{\per}{Per}
\DeclareMathOperator{\conv}{conv}

\begin{document}

\title{Piecewise rotations: limit set for the non -bijective maps}

\author{Nicolas B\'edaride\footnote{ Aix Marseille Université, CNRS, I2M UMR 7373, 13453 Marseille, France. Email: nicolas.bedaride@univ-amu.fr}\and  Jean-François Bertazzon\footnote{jeffbertazzon@gmail.com}\and Idrissa Kabor\'e\footnote{Universit\'e Nazi Boni,01 PB 1091, Bobo-Dioulasso, Burkina Faso. Email: ikaborei@yahoo.fr}}
\date{}

\maketitle

\begin{abstract}
We consider non-bijective piecewise rotations of the plane.  
These maps belong to a family introduced in previous papers by Boshernitzan and Goetz. 
We derive in this paper some upper bounds 
to the size of the limit set. This improves results of \cite{Bosh.Goet.03}.

\end{abstract}


\section{Introduction}
Consider a line $\mathbf D$ in $\mathbb{R}^2$, it splits the plane in two half-planes. Now define a dynamical system $T$ on $\mathbb{R}^2$ such that the restriction to each half-plane of $T$ is given by a euclidean rotation. The two rotations are of the same angle 
but with different centers, which may 
well lie outside of the corresponding half-plane. 

This type of map is called \emph{a piecewise isometry}. A lot of examples have been studied in the last years,
 see \cite{Goetz.00}, \cite{Ashwin.Fu} or \cite{Buzzi}. The maps studied here have the advantage of belonging 
to a family with a lot of interesting properties, see \cite{Bosh.Goet.03}, \cite{Goet.Quas.09} and \cite{Che.Goe.Qua.12} for references. 

Throughout 
this paper we follow the notation and conventions 
of Boshernitzan-Goetz \cite{Bosh.Goet.03}. They studied 
non-bijective maps. They prove that either every orbit is bounded uniformly by some constant $M$, or 
else every point outside a ball of radius $M$, has a divergent orbit.

The bijective case was not treated 
in this paper; it has been done by Goetz and Quas for a rational angle in the symmetric case 
(i.e. the middle of the segment between the two centers of rotations belongs to the line $\mathbf D$), see \cite{Goet.Quas.09}. In \cite{Bed.Kab.20} two authors of this paper give a complete description of the bijective symmetric map if the angle belongs to a 
specified finite set. In \cite{Bed.Kab.18} they describe the non-symmetric bijective maps, assuming the angle belongs to a particular finite set.

The aim of this paper is to find, for any non-bijective map in the family, an explicit bound for $M$, for every given angle. If the angle is rational, the strategy is to use the proof of \cite{Bosh.Goet.03} and improve it, where this is possible. For the irrational case, a complete different approach is presented.

\subsection{Outline of the paper}

In Section \ref{sec:definitions} we recall the definition of our maps, give some notations and introduce some tools used later. In Section \ref{sec:limit} we introduce the notion of limit sets which allows us to state correctly 
a previously known result, see Theorem \ref{thm:BG-rephrase}. Moreover we give some results on limit sets and periodic islands, see Proposition \ref{le:periodic_ba} and Proposition \ref{prop:perturbation}.

In Section \ref{se:irrational} we consider the case where the angle is an irrational multiple of  $\pi$ and we prove Theorem \ref{thm:sec-irrat}. 
Finally in Section \ref{se:rational} we consider the rational case $p/q$ where $q>2$ is an even number, see Theorem \ref{th:casrat}.

\section{Definitions}\label{sec:definitions}

\subsection{Definitions of a piecewise rotation}

Consider the euclidean space $\mathbb R^2$ equipped with an orientation. In the following we will use the classical identification between $\mathbb C$ and the euclidean plane. 
Let us fix a line $\mathbf D$. The line cuts the plane in two half planes, denoted 
by $\mathbf P_0$ and $\mathbf P_1$. 
Moreover we assume that $\mathbf P_1$ is closed, and that $\mathbf P_0$ is open, in order to have a partition of $\mathbb C$.

\begin{definition}
Now we fix a real number $\alpha$ and two points $C_0, C_1$. We define $\mathcal T$ as the set of maps $T$
$$\begin{array}{ccc}
\mathbb C&\rightarrow&\mathbb C\\
z&\mapsto& T(z) = e^{i\alpha}(z-C_j)+C_j\quad  \text{ if } \quad z\in \mathbf P_j, j\in\{0,1\}.
\end{array}$$
\end{definition}
\begin{remark}
The cases with $C_0=C_1$ or $\alpha=0$ 
will not be treated in the 
sequel. 
But the dynamic is obvious in these cases.
\end{remark}

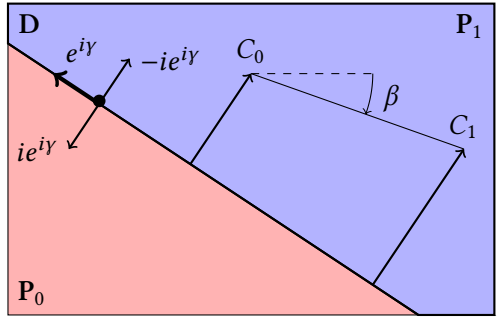
\begin{figure}
\centering		
	\begin{tikzpicture}[scale = 0.8]
	\draw[fill = red!30, draw] 
		(-3,-2.5) node[above right] {$\mathbf P_0$} --(-3,2)--(5,-2.5)-- cycle;
	\draw[fill = blue!30, draw, thick] 
		(-3,2.66)--(-3,2)--(3.75,-2.5)--(5,-2.5)--(5,2.66) node[below left]  {$\mathbf P_1$}-- cycle;
	\draw[thick] 
		(-3,2) -- (3.75,-2.5) node[above right] at (-3,2) {$\mathbf D$};
	\node at (-1.5,1) {\huge \textbullet};
	\draw[->,ultra thick,color = black] (-1.5,1)--(-2.25,1.5) node[above right] {$e^{i\gamma}$};
	\draw[->,thick,color = black] (-1.5,1)--(-2,0.25) node[left] {$ie^{i\gamma}$};
	\draw[->,thick,color = black] (-1.5,1)--(-1,1.75) node[right] {$-ie^{i\gamma}$}; ;
	\draw[->,thick,color = black] (0,0)--(1,1.5) node[above]  {$C_0$};
	\draw[->,thick,color = black] (3,-2)--(4.5,0.25) node[above] {$C_1$}; 
	\draw (1,1.5)--(4.5,0.25);
	\draw[dashed] (1,1.5)--(3,1.5);
	\draw [domain=0:-19.66,->] plot ({1+2*cos(\x)}, {1.5+2*sin(\x)});
	\node[right] at (3,1.1) {$\beta$};
	\end{tikzpicture}
\caption{
	Example of an application $T\in \mathcal T$ with
	$C_0 = 1+1.5i$, $C_1 = 4.5+0.25i$, $\gamma = \pi-\arctan(2/3)$.
}
\label{fig:defT}
\end{figure}

\begin{notations}
Let us fix some 
notation, 
see Figure \ref{fig:defT}. 
\begin{itemize}
\item We parametrize the line $\mathbf D$ by a point $z_0$ and an angle $\gamma$.
The angle $\gamma$ is 
chosen so 
that $z_0+ie^{i\gamma}$ belongs to $\mathbf P_0$.
\item We call $r_0, r_1$ the isometries of the plane which coincide with the restriction of $T$ to the two half planes: for all 
$ z\in \mathbb C$, $r_j(z) = e^{i\alpha}(z-C_j)+C_j$ with $j\in\{0,1\}$.
\item 
The two centers have complex coordinates given by $C_j = |C_j|e^{ic_j}$, and we denote $\beta=\arg(C_1-C_0)$. 
\item In the following we will need the 
quantity $\Delta = -2|C_1-C_0|\times \sin(\alpha/2) \cos(\gamma+\alpha/2-\beta)$.
\end{itemize}
\end{notations}

These notations shows that the map $T$ is given by some parameters. The following remark will explain how to reduce the number of such parameters.

\begin{remark}[Conjugacy] \label{conjugaison}
Let $T\in \mathcal T$, and consider the maps $f_0:z\mapsto \rho e^{i\theta}z+b$ and $f_1:z\mapsto \rho e^{i\theta}\bar z+b$ with $\rho \in \mathbb R_+^*$, $\theta\in \mathbb R$ and $b\in \mathbb C$.
Then we denote $T_{j} = f_j \circ T\circ f_j^{-1}$ for $j\in\{0,1\}$. This map is in $\mathcal T$, and we have the following parameters:
\[
\begin{array}{|c||c|c|c|c|c||c|}
\hline
T &\alpha&C_0&C_1&\beta&\gamma &\Delta \\
\hline
\hline
T_{0}  &\alpha&f_0(C_0)&f_0(C_1)&\theta+\beta &\theta+\gamma&\rho \Delta \\
\hline
T_{1} &-\alpha&f_1(C_0)&f_1(C_1)&\theta-\beta &\theta-\gamma&-\rho \Delta \\
\hline
\end{array}
\]
Thus we can always assume 
$0\in \mathbf D$, and the given data for $\gamma,\alpha, C_0, C_1$ determine 
the map $T$. We call these 
data 
the parameters of $T$.
\end{remark}

Now we recall 
a known result for the maps from $\mathcal{T}$, 
see \cite{Bosh.Goet.03}. 
For completeness we include a proof.

\begin{lemma}\label{prop:inj-surj}
$\;$
\begin{itemize}
\item A map $T\in\mathcal{T}$ is 
surjective 
if and only if $\Delta\geq 0$. In this case the strip $T(\mathbf P_0)\cap T(\mathbf P_1)$ has width $2\Delta$.
\item A map $T\in\mathcal{T}$ is 
injective 
if and only if $\Delta\leq 0$. In this case the set $\mathbb R^2\setminus (T(\mathbf P_0)\cup T(\mathbf P_1))$ is a strip of width $-2\Delta$.
\end{itemize}

\end{lemma}
\begin{proof}
The image of $\mathbf D$ by $r_0, r_1$ are two parallel lines. Thus the images of the two half planes intersect or not. It is enough to compute the following scalar product to conclude, see Figure \ref{fig:casinj}.
We have for all $z\in \mathbf D$:
\begin{align*}
\scal{ r_1(z-ie^{i\gamma}) - r_1(z) }{ r_0(z)-r_1(z) }
	&=\scal{ -e^{i\alpha} ie^{i\gamma}  }{ e^{i\alpha}(C_1-C_0)-(C_1-C_0) }\\
	&=\scal{ -e^{i\alpha} ie^{i\gamma}  }{ (e^{i\alpha}-1)(C_1-C_0) }\\
	&=\scal{ e^{i\alpha} ie^{i\gamma}  }{ 2i\sin(\alpha/2)e^{i\alpha/2}|C_1-C_0| e^{i\beta} }\\
	&=-2|C_1-C_0|\sin(\alpha/2) \scal{ e^{i\alpha} ie^{i\gamma}  }{ ie^{i\alpha/2}e^{i\beta} }=\Delta \qedhere
\end{align*}
\end{proof}


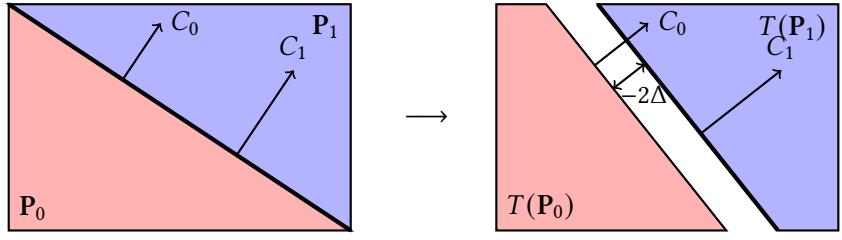
\begin{figure}
\begin{center}	 
	\begin{tikzpicture}[scale=0.5]
	\draw[fill = red!30, draw, thick] 
		(-3,-4)--(-3,2)--(6,-4)--(-3,-4) node[above right] {$\mathbf P_0$} -- cycle;
	\draw[fill = blue!30, draw, thick] 
		(-3,2)--(-3,2)--(6,2) node[below left] {$\mathbf P_1$} -- (6,-4) -- cycle;
	\draw[ultra thick] (-3,2)--(6,-4);
	\draw[->,thick,color = black] (0,0)--(1,1.5) node[right] at (1,1.5) {$C_0$};
	\draw[->,thick,color = black] (3,-2)--(4.5,0.25) node[above] at (4.5,0.25) {$C_1$}; 
	\node at (8,-1) {$\longrightarrow$};
	\end{tikzpicture}
\quad 
	\begin{tikzpicture}[scale=0.5]
	\draw[fill=red!30, draw, thick] 
		(-1.69,2)--(-3,2)--(-3,-4) node[above right] {$T(\mathbf P_0)$} --(3.05,-4) -- cycle;
	\draw[color=black] (-1.69,2)--(3.05,-4);
	\draw[fill=blue!30, draw, thick] 
		(-0.33,2)--(6,2) node[below left] {$T(\mathbf P_1)$} --(6,-4)--(4.41,-4) -- cycle;
	\draw[ultra thick] (-0.33,2)--(4.41,-4);
	\draw[->,thick,color = black] (-0.41,0.38)--(1,1.5) node[right]  {$C_0$}; ;
	\draw[->,thick,color = black] (2.38,-1.43)--(4.5,0.25) node[above] {$C_1$}; 
	\draw[<->,thick] (0.91,0.42)--(0.07,-0.24);
		\node[below right] at (0,0.1) {\small $-2\Delta$};
	\end{tikzpicture}
\end{center}
\caption{Injective example: same parameters as in Figure \ref{fig:defT} with an angle $\alpha = -\pi/10$.}
\label{fig:casinj}
\end{figure}

\begin{figure}
\begin{center}	 
	\begin{tikzpicture}[scale=0.5]
	\draw[fill = red!30, draw, thick] 
		(-3,-4)--(-3,2)--(6,-4)--(-3,-4) node[above right] {$\mathbf P_0$} -- cycle;
	\draw[fill = blue!30, draw, thick] 
		(-3,2)--(-3,2)--(6,2) node[below left] {$\mathbf P_1$} -- (6,-4) -- cycle;
	\draw[ultra thick] (-3,2)--(6,-4);
	\draw[->,thick,color = black] (0,0)--(1,1.5) node[right] at (1,1.5) {$C_0$};
	\draw[->,thick,color = black] (3,-2)--(4.5,0.25) node[above] at (4.5,0.25) {$C_1$}; 
	\node at (8,-1) {$\longrightarrow$};
	\end{tikzpicture}
\quad 
	\begin{tikzpicture}[scale=0.5]
	\draw[fill=red!30, draw, thick] 
		(-3,0.75)--(-3,-4) node[above right] {$T(\mathbf P_0)$}--(6,-4)--(6,-1.77) -- cycle;
	\draw[fill=blue!30,opacity=0.6, draw, thick] 
		(-3,-0.45)--(-3,2)--(6,2) node[below left] {$T(\mathbf P_1)$}--(6,-2.98) -- cycle;
	\draw[ultra thick] (-3,-0.45)--(6,-2.98);
	\node[below left] at (6,2) {$T(\mathbf P_1)$};
	\draw[->,thick,color = black] (0.51,-0.24) --(1,1.5) node[right]  {$C_0$};
	\draw[->,thick,color = black] (3.77,-2.35)--(4.5,0.25) node[below right] {$C_1$};  ; 
	\draw[<->] (1.09,-1.60) -- (1.40,-0.49);
		\node[right] at (1.2,-1.1) {\small $2\Delta$}; 
	\end{tikzpicture} 
\end{center}
\caption{Surjective example: same parameters as in Figure \ref{fig:defT} with an angle $\alpha = \pi/10$.}
\label{fig:cassurj}
\end{figure}
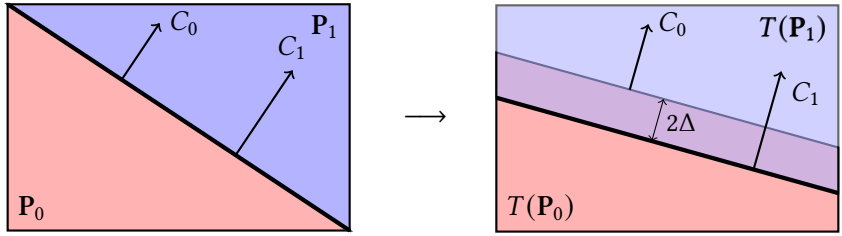

\subsection{Tools}

\begin{definition} \label{def-triple-norme}
Consider a map 
$T:\mathbb C\to \mathbb C$. We define:
\[
	\norme{T} = \sup\left\{ \big||T(z)|-|z|, z\in \mathbb C \right\} \in \mathbb R_+\cup\{+\infty\}.
\]
\end{definition}

\begin{lemma}\label{equ-rot}
For all $k\geq 0 $ and $z\in \mathbb C$ and for every map $T\in \mathcal T$ 
one has:
\begin{flalign*}
	&|T^k(z)-e^{ik\alpha}z|\leq k \norme{T}\\
	&\norme{T} \leq 2 \times | \sin(\alpha/2)| \times \max(|C_0|,|C_1|).
\end{flalign*}
\end{lemma}

\begin{proof}
For $z\in \mathbb C$ we have $T(z) = e^{i\alpha}(z-C_j)+C_j$, and thus:
\[
	\big| |T(z)| - |z| \big| =  \big| |T(z)| - |e^{i\alpha} z|\big|  
	\leq |T(z)-e^{i\alpha}z| = |C_j(1-e^{i\alpha})| 
	\leq 2|\sin(\alpha/2)|\max(|C_0|,|C_1|).
\]
Hence we have proved the formula for $k=1$ and $\norme{T} \leq 2|\sin(\alpha/2)|\max(|C_0|,|C_1|)$.

We obtain by triangular inequality
\begin{align*}
	|T^k(z)-e^{ik\alpha}z|
	&\leq \sum_{l=0}^{k-1} |e^{i(k-1-l)\alpha} T^{l+1}(z)-e^{i(k-l)\alpha} T^l(z)| \\
	&\leq \sum_{l=0}^{k-1} |T(T^l(z))-e^{i\alpha}T^l(z)|
	\leq \sum_{l=0}^{k-1}  \norme{T}
	\leq k  \norme{T}. \qedhere
\end{align*}
\end{proof}

\begin{lemma}\label{lem-ineg-module}
Let 
$y\in \mathbb C$ with $R>|y|$. Then we have:
\[ \big| |R+y|-R-\re(y)\big| \leq \frac{|y|^2}{R-|y|}.\]
\end{lemma}

\begin{proof}
Remark that $|R+y|^2 = R^2+2R\re(y)+|y|^2$ :
\[
|R+y|-R-\re(y)
	=\frac{|R+y|^2-R^2}{|R+y|+R}-\re(y)
	=\frac{|y|^2+2R \re(y)}{|R+y|+R}-\re(y)
	= \frac{|y|^2 +(R-|R+y|)\re(y)}{|R+y|+R}.
\]
If we denote $|R+y|-R-\re(y)= h$, then we have
$h= \frac{|y|^2 -(h+\re(y))\re(y)}{|R+y|+R}
	= \frac{\im^2(y)-\re(y)h }{|R+y|+R}$. 
We deduce $|h|\leq  \frac{|y|^2+|y|\cdot |h|}{R}$ and 
$0\leq |y|^2 + |h|(|y|-R)$, which gives the result.

\end{proof}

\begin{definition}\label{def-g}
Let us define the map $g$ :
\[
\begin{array}{ccc}
\mathbb R&\rightarrow&\mathbb R\\
x &\mapsto& g(x) = 
	\left\{
	\begin{array}{l}
	-2|C_0| \sin(\alpha/2)\sin(x+\alpha/2-c_0) \quad \text{if} \quad \gamma<x<\gamma+\pi \mod 2\pi .\\
	-2|C_1| \sin(\alpha/2)\sin(x+\alpha/2-c_1) \quad \text{otherwise}.
	\end{array}
	\right.
\end{array}
\]
\end{definition}

\begin{proposition} \label{propineg}
Consider a map $T\in\mathcal T$. Let $z=Re^{i\theta} \in \mathbb C$ 
such that $|z|> \norme{T}$, see Definition \ref{def-triple-norme}. 
Then we obtain the following inequalities, where $\theta_1$ is the argument of $T(z)$:
\begin{flalign}
	\label{majo1} 
	&\big| |T(z)|-|z|-g(\theta)\big| \leq \frac{\norme{T}^2}{|z|-\norme{T}},\\
	\label{majo3}
	&\norme{T}= 2|\sin(\alpha/2)|\max(|C_0|,|C_1|).\\
	\label{majo2} 
	&|\theta_1-(\theta+\alpha) \mod 2\pi |\leq \frac{\norme{T}}{|z|-\norme{T}}.
\end{flalign}
\end{proposition}
\begin{proof}

\begin{itemize}
\item \emph{Proof of \eqref{majo1} and \eqref{majo3}.}
Let $z = Re^{i\theta}\in \mathbb C$ such that $R>2|\sin(\alpha/2)|\max(|C_0|,|C_1|)$. 
Since $0\in \mathbf D=e^{i\gamma}\mathbb R$, we have
\[
z\in \mathbf P_0 \iff \gamma<\theta<\gamma+\pi \mod 2\pi.
\]
Let us denote $j\in\{0,1\}$ such that $z\in \mathbf P_j$. We obtain :
\[ 
	|T(z)|=|e^{i\alpha}(z- C_j) +C_j|
	=|Re^{i\theta}-C_j +e^{-i\alpha}C_j|
	=|R+C_je^{-i\theta}(e^{-i\alpha}-1)|
\]
Now we define 
\[ 
y = C_je^{-i\theta}(e^{-i\alpha}-1) = C_je^{-i(\theta+\alpha/2)}(e^{-i\alpha/2}-e^{i\alpha/2}) 
	=2\sin(\alpha/2)C_je^{-i(\theta+\alpha/2+\pi/2)}
\]
Note 
that $\re(y)=g(\theta)$, with $g$ defined in Definition \ref{def-g}. We apply Lemma \ref{lem-ineg-module} with $R=|z|$ and obtain:
\[
\big| |T(z)|-|z|-g(\theta)\big|  = \big| |R+y| - R - \re(y)\big| \leq \frac{|y|^2}{|z|-|y|}.
\]

For $|z|>|y|$ we 
deduce
\[
	\big| |T(z)| - |z| \big|  \geq |g(\theta)| - \frac{|y|^2}{|z|-|y|} .
\]
We conclude $\norme{T} \geq |y|$ 
using 
the fact that
$|y| = \sup|g(\theta)|$ and $\lim_{|z|\to+\infty}  \frac{|y|^2}{|z|-|y|}  = 0$. 
We conclude 
by means of 
Lemma \ref{equ-rot} 
to obtain 
Equation $(2)$. Then we deduce $(1)$.

\item \emph{Proof of \eqref{majo2}.} 
\[
	\left|\frac{T(z)e^{-i\alpha}}{z}-1\right| 
	=\left|\frac{T(z)}{z}-e^{i\alpha}\right| 
	=\left|\frac{e^{i\alpha}(z-C_j)+C_j}{z}-e^{i\alpha}\right| 
	= \frac{|C_j(1-e^{i\alpha})|}{|z|} \leq \frac{\norme{T}}{|z|}.
\]
If $\frac{\norme{T}}{|z|}<1$ we deduce : 
\begin{flalign*}
&\left|\im \left(\frac{T(z)e^{-i\alpha}}{z}\right)\right|
	= \left|\im \left(\frac{T(z)e^{-i\alpha}}{z}-1\right)\right|
	\leq  \left|\frac{T(z)e^{-i\alpha}}{z}-1 \right|
	\leq \frac{\norme{T}}{|z|}\\
&\left|\re \left(\frac{T(z)e^{-i\alpha}}{z}\right)\right|
	= 1+\re \left(\frac{T(z)e^{-i\alpha}}{z}-1\right)
	\geq 1-\frac{\norme{T}}{|z|}.
\end{flalign*}
Thus we obtain
\[
	|\tan(\theta_1-\alpha-\theta)| 
	= \frac{\left|\im \left(\frac{T(z)e^{-i\alpha}}{z}\right)\right|}{\left|\re \left(\frac{T(z)e^{-i\alpha}}{z}\right)\right|}
	\leq \frac{ \frac{\norme{T}}{|z|} }{ 1 - \frac{\norme{T}}{|z|} }
	= \frac{\norme{T}}{|z|-\norme{T}}.
\]

Thus we conclude
\[
	0\leq |\theta_1-(\theta+\alpha)| \leq  \arctan\left(\frac{\norme{T}}{|z|-\norme{T}}\right)
	\leq \frac{\norme{T}}{|z|-\norme{T}} .
	\qedhere
\]
\end{itemize}
\end{proof}

\section{Limit sets}\label{sec:limit}

\subsection{Definitions}

\begin{definition} 
Let $T: \mathbb C\to \mathbb C$ be an application.
For $M>0$, let  $B(0,M)$ be the open ball of radius $M$ centered at $0$. We define the following sets for any $n\in \mathbb N$:
\[
	\mathbf{A}_{M,n} (T) = T^nB(0,M)
	\quad \text{and} \quad
	\born_{M,n} (T) = T^{-n} B(0,M).
\]

\[
\attr(T)  = \bigcup_{M>0} \bigcap_{n\geq 0}\attr_{M,n}(T)
\quad \text{and} \quad
\born (T) = \bigcup_{M>0}  \bigcap_{n\geq 0} \born_{M,n}(T).
\]
\end{definition}

These sets will be called \emph{limit sets} in the rest of the paper.

\begin{remark}
Let $z\in \mathbb C$. The previous sets can be characterized as follows:
\begin{itemize}
\item 
$z\in \attr(T)\iff \exists (y_n)_n\in \mathbb C^{\mathbb N}$
a bounded sequence such that $\forall n\in \mathbb N$, $z = T^n(y_n)$.
\item  
$z\in \born(T)$ if and only if  $(T^n(z))_n$ is a bounded sequence.
\end{itemize}
\end{remark}

\begin{remark}\label{rem:imag}
By construction we obtain the following 
inclusions:
\begin{itemize}
\item 
$\forall n\in \mathbb N$, $\attr(T)\subset \attr(T^n)$ and $\born(T)\subset \born(T^n)$.
\item 
$\attr(T) \subset \bigcap_{n\geq 0} \im(T^n)$.
\end{itemize}
\end{remark}



\subsection{Global attraction and global repulsion}

\begin{definition} \label{def:attrrep}
Let $T$ be an application from $\mathbb C$ to $\mathbb C$. 
\begin{itemize}
\item 
The application $T$ is globally attractive if  there exists $M\geq 0$ such that $\forall L\geq 0$,
$\limsup\limits_{n\to+\infty}  \sup\limits_{|z|\leq L}  |T^n (z)| \leq M$.
\item 
The map $T$ is globally repulsive if there exists $M\geq 0$  such that $\forall L\geq M$,
$\liminf\limits_{n\to+\infty}  \inf\limits_{|z|\geq  L}  |T^n(z)|=+\infty$.
\end{itemize}
We will denote $M(T)$ the infimum of such real numbers $M$. 
\end{definition}

\begin{remark} \label{le:globorn}
Let $T$ be an application from $\mathbb C$ to $\mathbb C$.
\begin{itemize}
\item If $T$ is globally attractive, then $\attr(T)$ is bounded and included in $B(0, M(T))$.
\item If $T$ is globally repulsive, then, $\born(T)$ is bounded and included in $B(0,M(T))$.
\end{itemize}
\end{remark}

We are interested 
in the size of the limit sets.
Boshernitzan and Goetz have proved the following formula, reformulated here using our terminology. 

\begin{theoreme}\cite{Bosh.Goet.03}\label{thm:BG-rephrase}
\begin{itemize}
\item Let $T \in  \mathcal T$ 
be 
a surjective and non-injective map, then $T$ is globally attractive and
 $\attr(T)$ is a bounded set.
\item Let $T \in  \mathcal T$ 
be 
an injective and non-surjective map, then $T$ is globally repulsive and
 $\born(T)$ is a bounded set.
\end{itemize}
\end{theoreme}

We obtain an upper bound 
for $M(T)$, in the irrational and the rational case, by distinct methods presented in Sections \ref{se:irrational} and \ref{se:rational}.
 
\subsection{Periodic islands and limit sets} 

In this part we want to describe the limit sets 
$\born(T)$ and $\attr(T)$. 
Note 
that these sets can be 
quite complicated, 
as explained in \cite{Goetz.88} and \cite{Goetz2.88}. In Proposition \ref{le:periodic_ba} we describe some explicit subsets obtained by periodic islands (see Definition \ref{def:periodic_island}). 
Moreover, 
in Proposition \ref{prop:perturbation} and Corollary \ref{cor-attracteur-non-borne} we explain how these sets change when the parameters change, see Figure \ref{c-1}. 

%

Let us fix $\alpha$, $C_0$, $C_1$.
Now let $u=u_0,\ldots,u_{n-1}\in\{0,1\}^n$ be a finite word and consider the maps  $r_k(z) = e^{i\alpha}(z-C_k)+C_k$. 
We observe that 
\begin{equation}\label{eq:itere:rot}
r_{u_{n-1}} \circ \cdots \circ r_{u_0}(z) 
	= e^{in\alpha}z + \sum_{k=0}^{n-1} e^{-(n-k-1)i\alpha} C_{u_k}(1-e^{i\alpha}).
\end{equation}

\begin{definition}
Consider the word $u=u_0\ldots u_{n-1}\in\{0,1\}^n$ with $n\geq 1$ and assume $e^{in\alpha}\neq 1$. Then we define \emph{the almost periodic point associated to $u$} as the point
\[
z_u = \frac{1-e^{i\alpha}}{1-e^{in\alpha}} \sum_{k=0}^{n-1} e^{-(n-k-1)i\alpha} C_{u_k}.
\]
\end{definition}

The following lemma 
appears already 
in \cite{Goet.Quas.09} and \cite{Che.Goe.Qua.12}.

\begin{lemma}
Let $T\in \mathcal T$ and $u=u_0\ldots u_{n-1}\in\{0,1\}^n$ such that $e^{in\alpha}\neq 1$.
Then the almost periodic point $z_u$ is a periodic point of period $n$ for $T$ 
if and only if for all $k \in \llbracket 0,n-1\rrbracket$, $T^k(z_u) \in \mathbf P_{u_k}$.
\end{lemma}

A {\it coding} of the map $T$ is a map from $\mathbb{C}$ to $\{0,1\}^{\mathbb N}$ which sends the point $z$ to the sequence $(u_n)_{n\in\mathbb N}$ where $T^nz\in P_{u_n}$. 

\begin{definition} \label{def:periodic_island}
Let $n$ be such that $e^{in\alpha}\neq 1$ and consider a word $u=u_0\ldots u_{n-1}\in\{0,1\}^n$. 
Assume that $z_u$ is a periodic word for $T$ with coding $u^\omega$ (a periodic sequence of period $u$). Then we define the \emph{periodic island} as the set of complex numbers $z$ such that $T^k(z) \in \mathbf P_{u_{(k \mod n)}}$ for all integers $k\geq 0$. 

Moreover we define the \emph{weight} of $z_u$ 
as 
$w(z_u)=\min \big\{d(T^k(z),\mathbf D);k \in {\llbracket 0,n-1\rrbracket}\big\}.$

\end{definition}

Note that a periodic island is a convex set obtained as intersection of half-planes.
If $z_u$ is a periodic point, then $B(z_u,w(z_u))$ is included inside the periodic island associated to $z_u$, see \cite{Goet.Quas.09} 
(this inclusion becomes an equality 
if $\alpha$ is irrational, see Section 2 of \cite{Goet.Quas.09}).

\begin{proposition} \label{le:periodic_ba}
For 
any $T\in \mathcal T$ we have 
\[
\bigcup_{z_u \in \per} B(z_u,w(z_u)) \subset \born(T)
\quad\text{and}\quad
\bigcup_{z_u \in \per} B(z_u,w(z_u)) \subset 
\attr(T),
\]
where $\per = \{ z_u \text{ periodic point of some period $n$ such that }  e^{in\alpha}\neq 1\}$. 
\end{proposition}

\begin{proof}  
Let $z_u$ be 
a periodic point of period $n$ with  $e^{in\alpha}\neq 1$, and let $z\in B(z_u,w(z_u))$.
Notice 
that the restriction of $T^n$ to $B(z,w(z_u))$ is a rotation of angle $n\alpha$ 
with 
center $z_u$.
\begin{itemize}
\item 
Let  $k\in \mathbb N$ and consider the euclidean division: $k=qn+r$ with $0\leq r<n$. The point $T^{qn}(z)$ is in $B(z_u,w(z_u))$. By definition of $\norme{T}$ and Lemma \ref{equ-rot} we deduce
 
\[
|T^k(z)| = |T ^{r} \circ T^{qn}(z)|
	\leq r\norme{T} + |T^{qn}(z)|
	< n\norme{T} + |z_u|+ w(z_u).
\]
Thus we obtain $\sup\{|T^kz|\mid k\in \mathbb N\} < n\norme{T} + |z_u|+ w(z_u) <+\infty$ 
and $z\in \born(T)$.
\item 
The restriction of $T$ to $\cup_{i=0}^{n-1} T^{i} B(z_u,w(z_u))$ is a bijection 
onto its image, 
which is the same set. Let us denote 
this map by $\hat T$. Then 
for all $k\in \mathbb N$, if we denote
$y_k = {\hat T}^{-k} (z)$, we have $z = T^k(y_k)$ and we conclude $z\in \attr(T)$.
\end{itemize}
\end{proof}

\begin{proposition}
\label{prop:perturbation}
Let us consider 
an element $T$ of 
$\mathcal{T}$ with parameters $\alpha$, $\gamma$, $\mathbf D$ and $C_j$.
For all $\varepsilon>0$, let us denote 
by 
$T_\varepsilon$ the element of $\mathcal{T}$ with the same 
parameters, except for 
the discontinuity line $\mathbf D_\varepsilon = e^{i\varepsilon} \mathbf D$.
\newline
Consider 
a periodic point $z_u$ 
of $T$ of weight $w$ associated to the word $u$ of length $n$. Then for all $\varepsilon>0$ such that 
\[w_\varepsilon=w-2\big(|z_u| +n  \norme{T}  \big)\times |\sin (\varepsilon/2)|>0,\]
the point $z_u$ is a periodic point for $T_{\varepsilon}$ with weight bigger or 
equal to $w_\varepsilon$.
\end{proposition}

\begin{proof}
Let us consider $u=u_0\ldots u_{n-1}\in \{0,1\}^n$
and 
a periodic point $z_u$ 
for $T$ of weight $w$, associated to $u$.
By definition
we have:
\[
z_u = \frac{1-e^{i\alpha}}{1-e^{in\alpha}} \sum_{k=0}^{n-1} e^{-(n-k-1)i\alpha}  
C_{u_k}
\] 
Now we fix $\varepsilon$  small enough to have
\[
	w_\varepsilon=w-2\big(|z_u| +n \norme{T} \big)\times |\sin (\varepsilon/2)|>0.
\]
Let us prove that $z_u$ is a periodic point for $T_{\varepsilon}$.
We denote by 
$\mathbf P_{0}^\varepsilon$ and $\mathbf P_{1}^\varepsilon$ 
the half planes defined by $T_\varepsilon$.
\begin{itemize}
\item
We need to check that $k\in\llbracket 0,n-1\rrbracket$, $T_\varepsilon^k(z_u)\in \mathbf P_{u_k}^\varepsilon$
and $d(T_\varepsilon^k(z_u),e^{i\varepsilon} \mathbf D)\geq w_\varepsilon$. \newline
We'll give 
a proof by induction over $k$.
\item 
For $k=0$ we have by hypothesis
\[
d(e^{i\varepsilon}z_u,z_u) = |z_ue^{-i\varepsilon}-z_u|= 2|z_u| \times |\sin (\varepsilon/2)|.
\]
We deduce $z_u \in \mathbf P_{u_0}^\varepsilon$ 
and 
\begin{flalign*}
d(z_u,e^{i\varepsilon} \mathbf D)
	&\geq d(z_ue^{i\varepsilon},e^{i\varepsilon} \mathbf D) - d(z_u,e^{i\varepsilon} z)
		= d(z_u , \mathbf D) - d(z_u,e^{i\varepsilon} z)\\
	&>w - 2|z_u| \times|\sin (\varepsilon/2)| \geq w_\varepsilon.
\end{flalign*}
\item
Assume that the result is true for $k<n-1$. 
Then $T_\varepsilon^{k}(z_u)$ and $T^{k}(z_u)$ are in the same half planes 
$\mathbf P_{u_k}\cap \mathbf P_{u_k}^\varepsilon$ and 
$T^{k+1}_\varepsilon (z_u) = T^{k+1} (z_u)$ and 
then we deduce
\begin{flalign*}
d(T^{k+1} (z_u),e^{i\varepsilon}T^{k+1} (z_u))
	& = |T^{k+1} (z_u) -e^{i\varepsilon}T^{k+1} (z_u)|
		= |1-e^{i\varepsilon}| \times |T^{k+1} (z_u)| \\
	& \leq 2 |\sin (\varepsilon/2)|  \times (|z_u| + (k+1) \norme{T}).
\end{flalign*}
Thus we obtain
\begin{flalign*}
d(T^{k+1} (z_u),e^{i\varepsilon} \mathbf D)
	&\geq d(e^{i\varepsilon}T^{k+1} (z_u),e^{i\varepsilon} \mathbf D) 
		- d(T^{k+1} (z_u),e^{i\varepsilon} T^{k+1} (z_u))\\
	&= d(T^{k+1} (z_u) , \mathbf D) - d(T^{k+1} (z_u),e^{i\varepsilon}T^{k+1}  (z_u))\\
	&\geq w - 2 |\sin (\varepsilon/2)|  \times (|z_u| + (k+1) \norme{T}) \geq w_\varepsilon.
\end{flalign*}
We have proved our claim by induction.
\qedhere
\end{itemize}
\end{proof}

\begin{figure}
\begin{center}
\includegraphics[width=7cm]{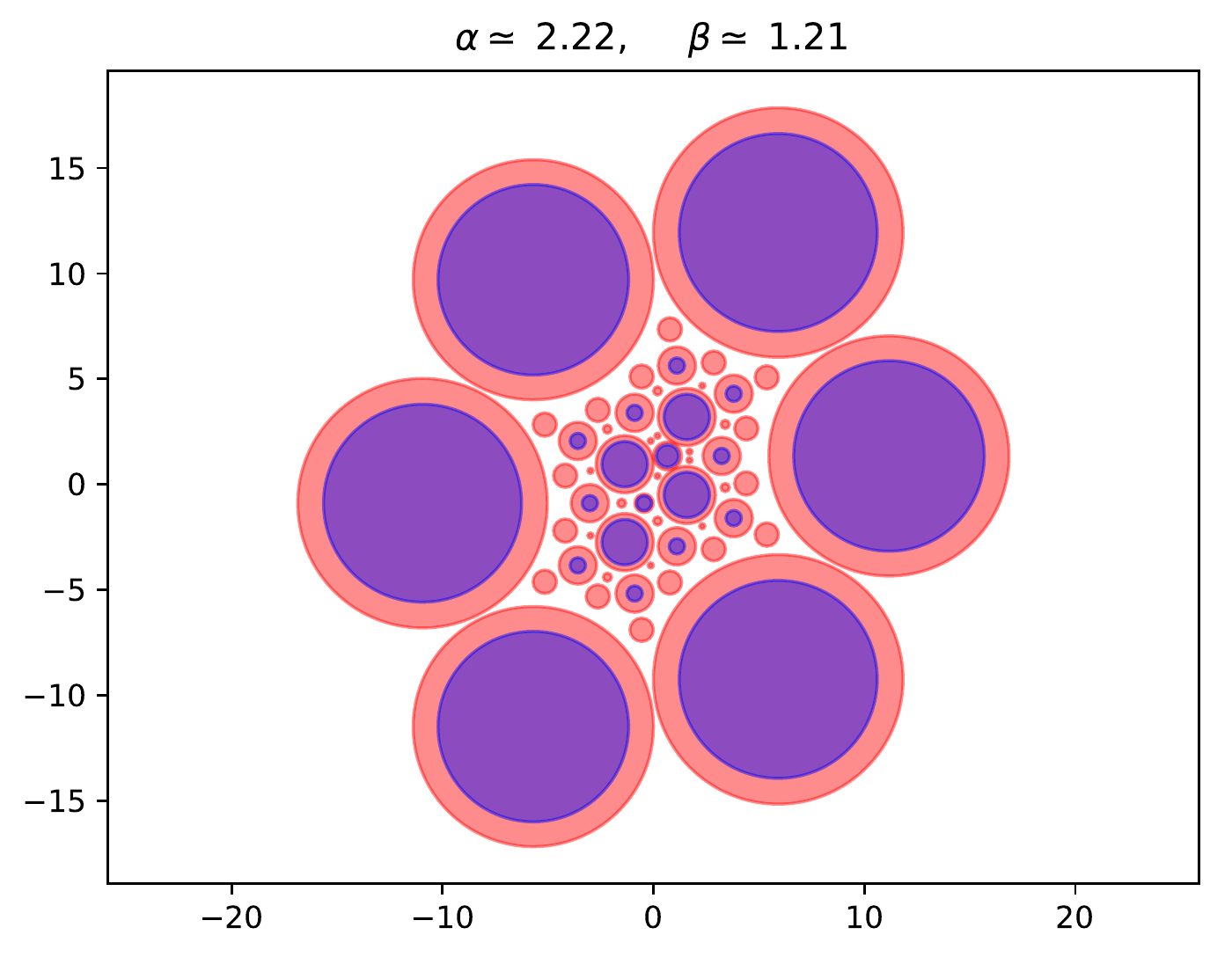}
\includegraphics[width=7cm]{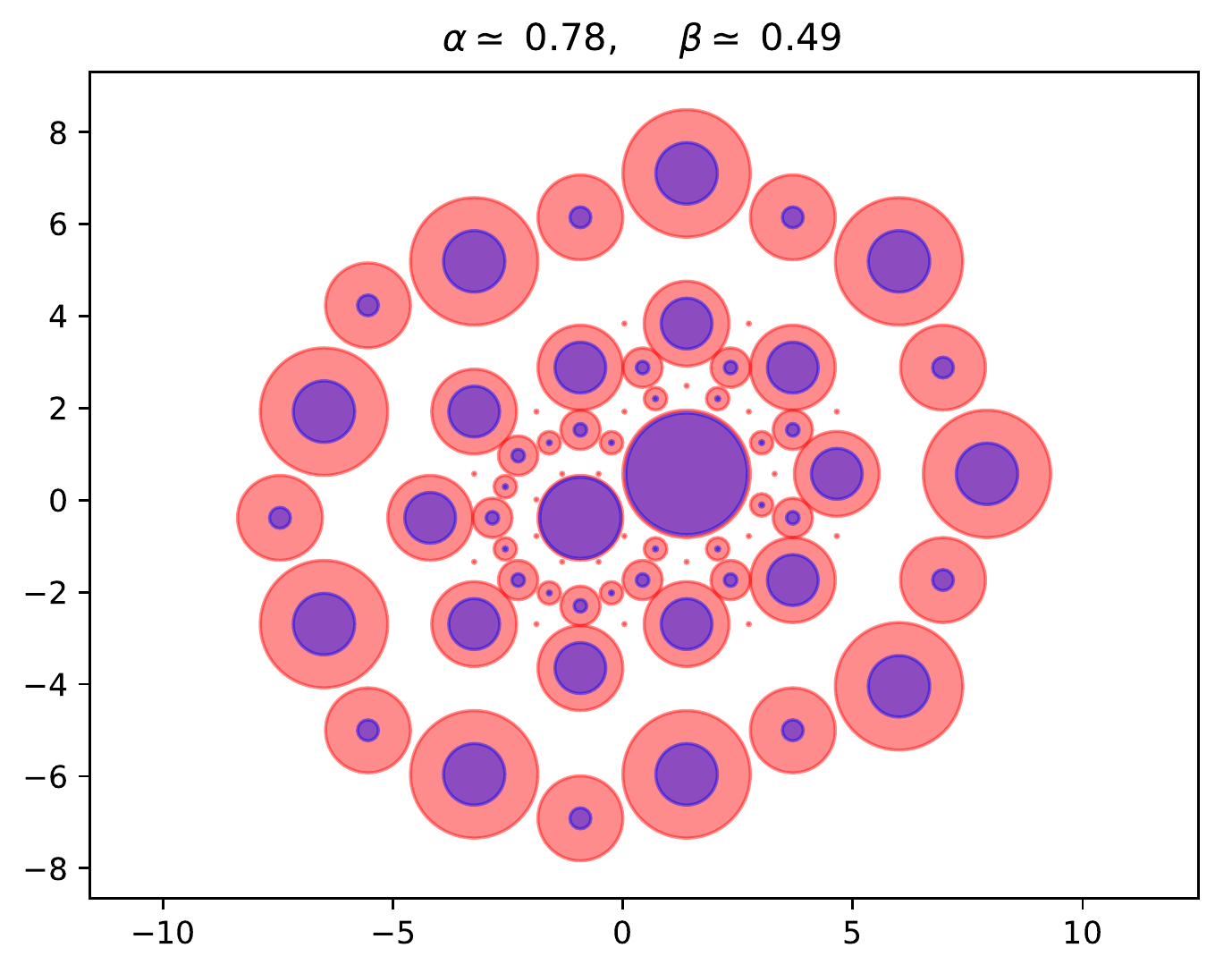}
\caption{Red part: periodic islands for a bijective case.	
	Blue part: periodic islands for a perturbation of the parameter $\gamma$, see
	Proposition \ref{prop:perturbation}.}
	\label{c-1}
\end{center}
\end{figure}

By 
using the conjugacy by $z\mapsto e^{i\varepsilon}z$ and Remark \ref{conjugaison} we can rephrase the previous proposition 
as follow:
\begin{proposition}
\label{co:perturbation} 
Let us consider 
the element $T$ 
of $\mathcal{T}$ with parameters $\alpha$, $\gamma$, $\mathbf D$ and $C_k$.
For all $\varepsilon$, let us denote 
by 
$T_\varepsilon$ the element of $\mathcal{T}$ with the same parameters but with centers $e^{i\varepsilon}C_k$. 

Consider 
a periodic point $z_u$ 
of $T$ of weight $w$ associated to the word $u$ of length $n$. Then for all $\varepsilon$ 
such that 
\[w_\varepsilon=w-2\big(|z_u| +n \norme{T} \big)\times |\sin (\varepsilon/2)|>
0
\]
we have that $e^{i\varepsilon}z_u$ is a periodic point for $T_{\varepsilon}$ 
with 
weight 
bigger or 
equal to $w_\varepsilon$.
\end{proposition}



We deduce:
\begin{corollary}\label{cor-attracteur-non-borne}
For all $M>0$, there exists a non-bijective piecewise rotation $T\in \mathcal{T}$ such that $\born$ or $\attr$ are non empty and contain a ball of radius $M$.
\end{corollary}
\begin{proof}
We apply Proposition \ref{co:perturbation} 
with $\beta=\alpha/2+\gamma-\pi/2$. In this case we have a bijective map, thus it has some periodic islands. 
Hence 
for some $\varepsilon>0$ we can find some non-bijective maps $T_\varepsilon$ such that  $\born$ or $\attr$ are non-empty and we apply Proposition \ref{le:periodic_ba} to conclude.
\end{proof}

\section{Size of the limit set for irrational angle} \label{se:irrational}

\subsection{Statement of the result}

Let 
$T = T(\alpha,C_0,C_1,\mathbf D, \gamma)\in \mathcal T$ 
be
\textbf{a  non-bijective map}. Let us denote $\alpha=2\pi a$ and assume 
that 
$a$ is an irrational number. Recall that $M(T)$ is defined in Definition \ref{def:attrrep}.

Let 
$(p_\ell)_{\ell\in \mathbb N}$ and $(q_\ell)_{\ell\in\mathbb N}$ 
be 
the sequence of convergents of $a$ in the continued fraction expansion.
Let $\ell_0$ be such that
$\frac{1}{q_{\ell_0}} \left(  \left(8+ \frac{\pi}{2q_{\ell_0}} \right) \norme{T} +4 |\Delta|  \right) < \frac{|\Delta|}{\pi}$. Then we will prove
the following:

\begin{theoreme}\label{thm:sec-irrat}
Let $T\in\mathcal T$ be a non-bijective map with an irrational angle.
We obtain the following bound:
$$M(T)\leq q_{\ell_0} \norme{T} \left(\frac2{\pi} q_{\ell_0}  +1\right)$$
\end{theoreme}

To prove this theorem we first need some background on 
the Denjoy-Kosma inequality and the Three gaps theorem. 

\subsection{Bounded variation function and three gaps theorem}

\begin{definition}
Consider a periodic function $f:\mathbb R\to \mathbb R$ of  period $T$. If the following quantity exists 
we say that $f$ has bounded variation with variation $\Var(f)$:

\[
	\Var(f) = \sup\left\{ \sum_{k=0}^{P-1} |f(x_{k+1})-f(x_k)| ;  \{x_0,\ldots,x_{P}\} \in \mathcal P\right\}.
\]
where $\mathcal P$ is the set of subdivisions $ P = \{x_0,\ldots,x_{P-1}\}$ of $[0,T]$
such that
\[
	0=x_0\leq x_1\leq \cdots \leq x_{P-1} \leq x_{P}=T .
\]

\end{definition}

\begin{proposition}\label{prop-var}
Let $f:\mathbb R\to \mathbb R$ be a $2\pi$-periodic 
function
which is 
piecewise $\mathcal C^1$.
Assume that the discontinuity points of $f$ on $]0,2\pi[$ are denoted 
by $d_1,\ldots,d_m$, and set $d_0=0$. Then 
$f$ has bounded variation and
\[
	\Var(f) = \int_0^{2\pi} |f'(t)| dt + \sum_{k=0}^m |  f(d_k^+)-  f(d_k^-)|.
\]
\end{proposition}

We associate to any irrational positive real number 
$0<a<1$ 
its continued fraction expansion:
\[
a = 0 + \frac{1}{\displaystyle a_1 
          + \frac{1}{\displaystyle a_2 + \frac{1}{\ddots}}}
\mbox{ that we denote by $ a = [ 0 ,a_1, \dots , a_\ell , \dots ]$. }
\]
We also consider two sequences of integers $(p_k)_{k\in\mathbb N}$
and $(q_k))_{k\in\mathbb N}$, such that for any integer 
$\ell \geq 0$ one has 
$[ 0,a_1, \dots , a_\ell] = \frac{p_\ell}{q_\ell}$.
They are defined by induction as follows,

\[
\left\{ 
\begin{array}{ll}
p_{-2} = 0 \\
q_{-2} = 1
\end{array}
\mbox{  ,  }
\begin{array}{ll}
p_{-1} = 1 \\
q_{-1} = 0
\end{array}
\right.
\mbox{ and for any integer }\ell \geq0:
\left\{
\begin{array}{ll}
p_{\ell} = a_\ell p_{\ell-1}  + p_{\ell-2} .\\
q_{\ell} = a_\ell q_{\ell-1}  + q_{\ell-2} .
\end{array}
\right.
\]

We recall Denjoy–Koksma inequality:
\begin{theoreme}\cite{Kuip.Nied.74}\label{thm-DK}
Let $a\in[0,1[$ be an irrational number, 
and 
let us denote 
by 
$\frac{p_\ell}{q_\ell}$ %
the sequence of partial quotients in the continued fraction expansion of $a$. Let $f:\mathbb R\to \mathbb R$ be a $2\pi$-periodic function, piecewise  $\mathcal C^1$. 
Then for all $x\in \mathbb R$ 
and for all $\ell\in \mathbb N^*$
one has:
\[
	\left| \frac{1}{q_\ell}\sum_{k=0}^{q_\ell-1} f(x+2\pi ka) - \frac{1}{2\pi} \int_{0}^{2\pi} f(t)dt \right| 
		\leq \frac{1}{q_\ell}\Var (f). 
\]
\end{theoreme}

The following theorem has been proved by V. Sos: 
\begin{theoreme}[Three gaps theorem]\cite{Aless.Bert.98}\label{thm-3l}
Let $0 < a < 1$ be an irrational number and $\ell$ a positive integer. 
Let us denote 
by $p_\ell/q_\ell$ a partial quotient.
The points $\{ \{ka\}, 0 \leq k \leq q_\ell-1\}$
partition the unit circle into $q_\ell$ intervals. 
The lengths of 
the latter take on 
three values, 
and the minimal value is equal to $|q_{\ell-1} a-p_{\ell-1}|\geq \frac{1}{2q_\ell}$.
\end{theoreme}

\subsection{Auxiliary function}

\begin{lemma}\label{lem-g}
Consider the map $g$
from 
Definition \ref{def-g}. We have the following properties:
\begin{itemize}
\item The map $g$ is of bounded variation and we have $\Var(g) \leq 2|\Delta|+4 \norme{T}$.
\item $\displaystyle\int_{-\pi}^{\pi} g(x) d x = -2\Delta$.
\end{itemize}
\end{lemma}

\begin{proof}
$\;$
\begin{itemize}
\item The map $g$ has two 
discontinuities, in $\gamma$ and $\gamma+\pi$,
and we have
\begin{flalign*}
&|g(\gamma^+)-g(\gamma^-)|+|g(\gamma+\pi^+)-g(\gamma+\pi^-)|\\
&\quad = 2|\sin(\alpha/2)| \times \big| |C_0| \sin(\gamma+\alpha/2-c_0)-|C_1| \sin(\gamma+\alpha/2-c_1)\big|\\
&\quad \quad +2|\sin(\alpha/2)| \times \big| |C_0| \sin(\gamma+\pi+\alpha/2-c_0)-|C_1| \sin(\gamma+\pi+\alpha/2-c_1)\big|\\
&\quad = 2|\sin(\alpha/2)| \times \big| \im(\bar C_0e^{i(\gamma+\alpha/2)})
	-\im(\bar C_1e^{i(\gamma+\alpha/2)})\big|\\
&\quad \quad + 2|\sin(\alpha/2)| \times \big| \im(\bar C_0e^{i(\gamma+\pi+\alpha/2)})
	-\im(\bar C_1e^{i(\gamma+\pi+\alpha/2)})\big|\\
&\quad = 4|\sin(\alpha/2)| \times |C_0-C_1| \times |\sin(\gamma + \alpha/2 - \beta)|\\
&\quad = 2 |\Delta|.
\end{flalign*}

We obtain for the computation of the integral
\begin{flalign*}
&\int_{\gamma-\pi}^{\gamma} |g'(x)| d x+\int_{\gamma}^{\gamma+\pi} |g'(x)| d x\\
	&=2|C_1| |\sin(\alpha/2)|\int_{\gamma-\pi}^{\gamma} |\cos(x+\alpha/2-c_1)| d x
		+2|C_0| |\sin(\alpha/2)| \int_{\gamma}^{\gamma+\pi} |\cos(x+\alpha/2-c_0)|d x\\
	&= 4|\sin(\alpha/2)|(|C_0|+|C_1|) \leq 8|\sin(\alpha/2)|\max(|C_0|,|C_1|).
\end{flalign*}
We use \eqref{majo1} in Proposition \ref{propineg} and Proposition \ref{prop-var} to conclude at the first point.

\item We obtain:
\begin{flalign*}
\int_{-\pi}^{\pi} g(x) d x
	&= \int_{\gamma-\pi}^{ \gamma} g(x) d x+\int_{\gamma}^{\gamma+\pi} g(x) d x\\
	&=-2|C_1| \sin(\alpha/2)\int_{\gamma-\pi}^{\gamma} \sin(x+\alpha/2-c_1) d x
		-2|C_0| \sin(\alpha/2) \int_{\gamma}^{\gamma+\pi} \sin(x+\alpha/2-c_0)d x\\
	&=2|C_1|\sin(\alpha/2) \left[ \cos(x+\alpha/2-c_1)  \right]_{\gamma-\pi}^{\gamma} 
	+2|C_0|\sin(\alpha/2) \left[ \cos(x+\alpha/2-c_0)  \right]_{\gamma}^{\gamma+\pi} \\
	&=4|C_1|\sin(\alpha/2)  \cos(\gamma+\alpha/2-c_1) 
		- 4|C_0|\sin(\alpha/2)  \cos(\gamma+\alpha/2-c_0)   \\	
	&=4|C_1-C_0|\sin(\alpha/2)  \cos(\gamma+\alpha/2-\beta)  
\end{flalign*}
where 
we use the fact that
$\re(\bar C_1e^{i(\alpha/2+\gamma)}) - \re(\bar C_0e^{i(\alpha/2+\gamma)}) 
	= \re\left(\overline{(C_1-C_0)}e^{i(\alpha/2+\gamma)}\right)$.
\end{itemize}
\end{proof}

Now we state the key part in the proof of Theoem \ref{thm:sec-irrat}.

\begin{proposition}\label{prop1}
Consider $\ell \in \mathbb N, z\in \mathbb{C}$ and 
assume $|z|> q_\ell \norme{T} \left(\frac2{\pi} q_\ell  +1\right)
$. Then we obtain:
\[
	\left| \frac{|T^{q_\ell} z| - |z|}{q_\ell} + \frac{\Delta}{\pi}  \right|
	\leq
	\frac{1}{q_\ell} \left(  (8+ \frac{\pi}{2q_\ell}) \norme{T} +4 |\Delta|  \right).
\]

\end{proposition}

\begin{proof}
We will need the following notations.
If $z\in \mathbb C$, then we denote 
by 
$\theta$ its 
argument, 
and for all integer 
$k\geq 0$ we set 
$z_k := T^k(z)$ and $\theta_k := \arg(z_k)$. Moreover we will denote, by abuse of notations, $d(\theta,\theta')=|\theta-\theta' \mod 2\pi|$. 
We obtain:

\begin{flalign*}
\left| \frac{|T^{q_\ell} z| - |z|}{q_\ell} + \frac{\Delta}{\pi} \right|
	&= \left| \frac{1}{q_\ell} \sum_{k=0}^{q_\ell-1} \left(|z_{k+1}|-|z_k| \right) + \frac{\Delta}{\pi} \right|\\
	& \leq  \frac{1}{q_\ell} \sum_{k=0}^{q_\ell-1} \left| |z_{k+1}|-|z_k|-g(\theta_k )\right| 
		+ \left|\frac{1}{q_\ell} \sum_{k=0}^{q_\ell-1} g(\theta_k )  + \frac{\Delta}{\pi} \right|.
\end{flalign*}

By Lemma \ref{equ-rot} and 
equation 
\eqref{majo3} we obtain, if $|z|>\norme{T}q_{\ell}$, that $|z_k|>\norme{T}$ for $0\leq k\leq q_\ell-1$.\newline
Then by 
equation 
\eqref{majo1} we have: 

\begin{flalign*}
&\left| \frac{|T^{q_\ell} z| - |z|}{q_\ell} + \frac{\Delta}{\pi} \right|
	\leq  \frac{1}{q_\ell} \sum_{k=0}^{q_\ell-1}  \frac{\norme{T}^2}{|z_k|-\norme{T}}  
	+ \left|\frac{1}{q_\ell} \sum_{k=0}^{q_\ell-1} g(\theta_k )  + \frac{\Delta}{\pi} \right|\\
&\qquad
	 \leq  
	\underbrace{\frac{1}{q_\ell} \sum_{k=0}^{q_\ell-1}  \frac{\norme{T}^2}{|z_k|-\norme{T}} }_{(A)} 
	+
	\underbrace{ \left|\frac{1}{q_\ell} \sum_{k=0}^{q_\ell-1} g(\theta_k )  - g(\theta_0+k\alpha) \right|}_{(B)}
	 + 
	 \underbrace{\left|\frac{1}{q_\ell} \sum_{k=0}^{q_\ell-1} g(\theta_0+k\alpha)  + \frac{\Delta}{\pi} \right|}_{(C)}.
\end{flalign*}

We will produce upper bound for each term.
\begin{itemize}
\item[$(A)$] 
First 
term:
\[
	\frac{1}{q_\ell} \sum_{k=0}^{q_\ell-1}  \frac{\norme{T}^2}{|z_k|-\norme{T}}  
	\leq \frac{1}{q_\ell} \sum_{k=0}^{q_\ell-1}  \frac{\norme{T}^2}{|z|-(k+1)\norme{T}}  
	 \leq  \frac{\norme{T}^2}{|z|-q_\ell\norme{T}}   .
\]
\item[$(B)$] Second term:
\begin{itemize} %
\item %
By 
equation 
\eqref{majo2} in Proposition \ref{propineg}: if $|z|>\norme{T}$, we obtain
$d(\theta_1,\theta_0+\alpha) \leq \frac{\norme{T}}{|z|-\norme{T}}$.
\item
If $|z|>2\norme{T}$, then $|z_1| = |T(z)| \geq |z|- \norme{T}> \norme{T}$ and we deduce
\begin{flalign*}
d(\theta_2,\theta_0+2\alpha) 
	&\leq d(\theta_{2},\theta_1+\alpha) +d(\theta_1+\alpha,\theta_0+2\alpha)\\
	&\leq d(\theta_{2},\theta_1+\alpha) +d(\theta_1,\theta_0+\alpha)\\
	&\leq  \frac{\norme{T}}{|z_1|-\norme{T}}	+  \frac{\norme{T}}{|z|-\norme{T}} \\
	&\leq  \frac{\norme{T}}{|z|-2\norme{T}}	+  \frac{\norme{T}}{|z|-\norme{T}}	
\end{flalign*}
\item
Now if $|z|>n\norme{T}$, we have $\forall k \in \llbracket 0,n-1 \rrbracket$, $|z_k|>\norme{T}$ and
\begin{flalign*}
d(\theta_n,\theta_0+n\alpha) 
	&\leq \sum_{k=0}^{n-1} d(\theta_{n-k}+k\alpha,\theta_{n-k-1}+(k+1)\alpha) \\
	&\leq \sum_{k=0}^{n-1} d(\theta_{n-k},\theta_{n-k-1}+\alpha) \\
	&\leq \sum_{k=0}^{n-1}  \frac{\norme{T}}{|z_k|-\norme{T}} \\
	&\leq \sum_{k=1}^{n}  \frac{\norme{T}}{|z|-k\norme{T}} \\
	&\leq  \frac{n\norme{T}}{|z|-n\norme{T}} \\
\end{flalign*}
\item
By Theorem \ref{thm-3l} the smallest interval in the subdivision
$\{k\alpha, 0\leq k \leq q_\ell-1\}$ has 
length at least $\frac{2\pi}{2q_\ell}=\frac{\pi}{q_\ell}$.

\item we 
observe 
the following fact:
\begin{flalign*}
(q_\ell-1)  \frac{\norme{T}}{|z|-(q_\ell-1)\norme{T}} < \frac{\pi}{2q_\ell}
	&\iff 
		\frac2{\pi} q_\ell (q_\ell-1)  \norme{T} < |z|-(q_\ell-1)\norme{T}\\
	&\iff 
		|z| > (q_\ell-1)  \norme{T} \left(\frac2{\pi} q_\ell  +1\right).		
\end{flalign*}

Thus if $|z|> (q_\ell-1)  \norme{T} \left(\frac2{\pi} q_\ell  +1\right)$ and $|z|>q_\ell \norme{T}$  then 
for all $0\leq k \leq q_\ell-1$, we obtain
$d(\theta_k,\theta_0+k\alpha) < \frac{\pi}{2q_\ell}$.

\begin{figure}
\begin{center}
\begin{tikzpicture}
\def\taille{2}
\draw (0,0) circle (\taille) ;
\foreach \x in {0,1,...,7}{
	{\pgfmathparse{\taille*cos(\x*360*1.618)}\let\y\pgfmathresult
	{\pgfmathparse{\taille*sin(\x*360*1.618)}\let\z\pgfmathresult
		 \node[rotate=\x*360*1.618,blue] at (\y,\z) {$\blacksquare$} ;
		 \node[right,rotate=\x*360*1.618] at (\y,\z) {$\theta_0+\x\alpha$} ;}}}
\foreach \x in {0,1,...,7}{
	{\pgfmathparse{\taille*cos(\x*360*1.618+(sin(\x)+2)*5 )}\let\y\pgfmathresult
	{\pgfmathparse{\taille*sin(\x*360*1.618+(sin(\x)+2)*5 )}\let\z\pgfmathresult
		 \node[red] at (\y,\z) {$\bullet$} ;
		 \node[right,rotate=\x*360*1.618+(sin(\x)+2)*10] at (\y,\z) {$\theta_{\x}$} ;}}}
\end{tikzpicture}
\end{center}
\caption{
The blue points are the $\{k \alpha\}$ with $0\leq k \leq 7$ and $\alpha = (1+\sqrt 5)/2$. The red points represent the $\theta_k$.
}
\end{figure}
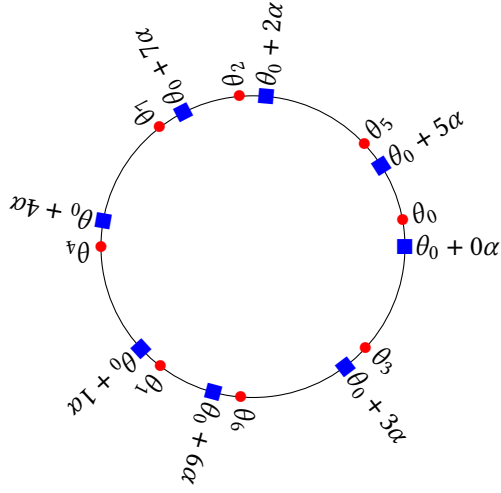
\item
Thus if $|z|> q_\ell \norme{T} \left(\frac2{\pi} q_\ell  +1\right)$, we obtain for the second term:
\[
	\left|\frac{1}{q_\ell} \sum_{k=0}^{q_\ell-1} g(\theta_k )  - g(\theta_0+k\alpha) \right|
	\leq 
	\frac{1}{q_\ell} \Var(g) .	
\]
\end{itemize} %
\item[$(C)$] 
Third term: we use 
the 
Denjoy–Koksma inequality (Theorem \ref{thm-DK}) and Lemma \ref{lem-g}
to deduce:
\[
	\left|\frac{1}{q_\ell} \sum_{k=0}^{q_\ell-1} g(\theta_0+k\alpha)  + \frac{\Delta}{\pi} \right|
	=
	\left|\frac{1}{q_\ell} \sum_{k=0}^{q_\ell-1} g(\theta_0+k\alpha)  
		- \frac{1}{2\pi}\int_{-\pi}^\pi g(x)dx \right|
	\leq 
	\frac{1}{q_\ell} \Var(g) .
\]
\end{itemize}

Thus if we have $|z|> q_\ell \norme{T} \left(\frac2{\pi} q_\ell  +1\right)$,
we deduce 
from Lemma \ref{lem-g} that

\begin{flalign*}
	\left| \frac{|T^{q_\ell} z| - |z|}{q_\ell} + \frac{\Delta}{\pi}  \right|
	&\leq   
\frac{\norme{T}^2}{|z|-q_\ell\norme{T}}   + \frac{2}{q_\ell} \Var(g) \\
	&\leq   
	\frac{\norme{T}^2}{q_\ell \norme{T} \left(\frac2{\pi} q_\ell  +1\right)-q_\ell\norme{T}}   
	+ \frac{2}{q_\ell} \left( 2|\Delta|+4 \norme{T} \right) \\
	&\leq\frac{1}{q_\ell} \left(  \frac{\norme{T}}{\frac2{\pi} q_\ell} +4 |\Delta| +  8 \norme{T}) \right)
	=\frac{1}{q_\ell} \left(  \left(8+ \frac{\pi}{2q_\ell} \right) \norme{T} +4 |\Delta|  \right)
\end{flalign*}
which proves the result.
\end{proof}

\subsection{Conclusion of the proof}

Let $\ell_0$ such that
$\frac{1}{q_{\ell_0}} \left(  \left(8+ \frac{\pi}{2q_{\ell_0}} \right) \norme{T} +4 |\Delta|  \right) < \frac{|\Delta|}{\pi}$ 
and let us denote 
$\epsilon =  \frac{|\Delta|}{\pi} - \frac{1}{q_{\ell_0}} \left(  \left(8+ \frac{\pi}{2q_{\ell_0}} \right) \norme{T} 
	+4 |\Delta|  \right).$

\begin{itemize}
\item
If $\Delta>0$, then by Proposition \ref{prop1} for
$|z|> q_{\ell_0} \norme{T} \left(\frac2{\pi} q_{\ell_0}  +1\right)$
one has:
\[
|T^{q_{\ell_0}} z| \leq |z| - q_{\ell_0} \epsilon
\]
Thus we obtain
\[
	\attr(T^{q_{\ell_0}}) 
	\subset B\left(0,q_{\ell_0} \norme{T} \left(\frac2{\pi} q_{\ell_0}  +1\right) -q_{\ell_0}\epsilon\right)
	\subset B\left(0, q_{\ell_0} \norme{T} \left(\frac2{\pi} q_{\ell_0}  +1\right) \right).
\]

By definition of $\attr(T)$, we obtain 
(see 
Remark \ref{rem:imag}):
\[
	\attr(T) \subset \attr(T^{q_{\ell_0}}) 
	\subset
	B\left(0, q_{\ell_0} \norme{T} \left(\frac2{\pi} q_{\ell_0}  +1\right) \right).
\]
\item
If $\Delta<0$, then by Proposition \ref{prop1} we obtain for 
$|z|>q_{\ell_0} \norme{T} \left(\frac2{\pi} q_{\ell_0}  +1\right)$ :
\[
|T^{q_{\ell_0}} z| \geq |z| + q_{\ell_0} \epsilon
\]

Thus for every integer $n$ we deduce
\[
|T^{nq_{\ell_0}} z| \geq |z| + nq_{\ell_0} \epsilon
\]
and $z\notin B\left(0,q_{\ell_0} \norme{T} \left(\frac2{\pi} q_{\ell_0}  +1\right)\right)$.
Then by Remark \ref{rem:imag} we conclude:
\[
\born(T) \subset B\left(0,q_{\ell_0} \norme{T} \left(\frac2{\pi} q_{\ell_0}  +1\right)\right).
\]
\end{itemize}


\subsection{Example} \label{exemple_cas_irrationnel}

Consider $T$ with the following parameters :
\[
\renewcommand{\arraystretch}{1.5}
\begin{array}{|c|c|c|c|c|c||c|c|c|c|}
\hline
\alpha & a & C_0 & C_1  & \gamma & \mathbf D &
	\beta & \Delta& \norme{T} \\
\hline 
\frac{\sqrt{2}}{2} \pi& \frac{\sqrt{2}}{4}&-e^{1.14i}& 1.5e^{1.14i}& \frac{\pi}{2}& i\mathbb R &
	1.14 & -0.13>\Delta >-0.14 & 2.68<|3\sin(\alpha/2)|\leq 2.69 \\
\hline 
\end{array}
\]

Since $\Delta<0$ the map $T$ is injective and is not bijective.

Notice 
that the continued fraction of $a$ is equal to $a= [0,2,1,4,1,4,1,4,\ldots]$
\[
\begin{array}{|c||c|c|c|c|c|c|c|c|c|c|c|}
\hline
\ell & 0 & 1& 2& 3& 4& 5& 6&7&8& 9 \\
\hline
\hline 
p_\ell&0&1&1&5&6&29&35&169&204&985\\
\hline 
q_\ell&1&2&3&14&17&82&99&478&577&2786\\
\hline
\end{array}
\]
We obtain 
\[
\frac{1}{q_{8}} \left(  \left(8+ \frac{\pi}{2q_{8}} \right) \norme{T} +4 |\Delta|  \right)
<
0.0383
<
\frac{|\Delta|}{\pi}
< 
0.0460
<
\frac{1}{q_{7}} \left(  \left(8+ \frac{\pi}{2q_{7}} \right) \norme{T} +4 |\Delta|  \right)
.
\]
and, see Figure \ref{c-2}, we deduce the bound
$M = q_{8} \norme{T} \left(\frac2{\pi} q_{8}  +1\right) <571283.$ 
Notice that the real size of the set seems to be much smaller.

In fact we can obtain a better bound due to the following remark:
\begin{remark} \label{rem:amelioration}
The bound from $M$ can be improved. In order to avoid technicalities we only explain this for the above 
example.
Consider some conjugation in order to change the origin: let us define $p$ as the point of $\mathbf D$ which minimizes both distances to 
$C_0$ and $C_1$. 
This point exists by a compactness argument. If $p$ is the new origin, then we obtain new coordinates for the centers of rotations, and $\norme{T'}\leq \norme{T}$. In this example we obtain:

The point $p$ is at the intersection of $\mathbf D$ 
with the bisection line 
of $[C_0,C_1]$. This point is 
given by 
$p = \frac{0.25}{\sin(1.14)} i\simeq 0.275i.$
The new centers of rotations are $C'_0 = C_0 -p
	\quad \text{and} \quad
	C'_1 = C_1-p.$ Then the new map fulfills $\norme{T'} = 2\max(|C'_0|,|C'_1|) \sin(\alpha/2) \simeq 2.249.$ We obtain
$\frac{1}{q_{8}} \left(  \left(8+ \frac{\pi}{2q_{8}} \right) \norme{T'} +4 |\Delta|  \right)
<
\frac{|\Delta|}{\pi}$. 
We deduce the new bound
\[
M' =q_{8} \norme{T'} \left(\frac2{\pi} q_{8}  +1\right) <536048<M. 
\]
\end{remark}

\begin{figure}
\begin{center}
\includegraphics[width=7.5cm]{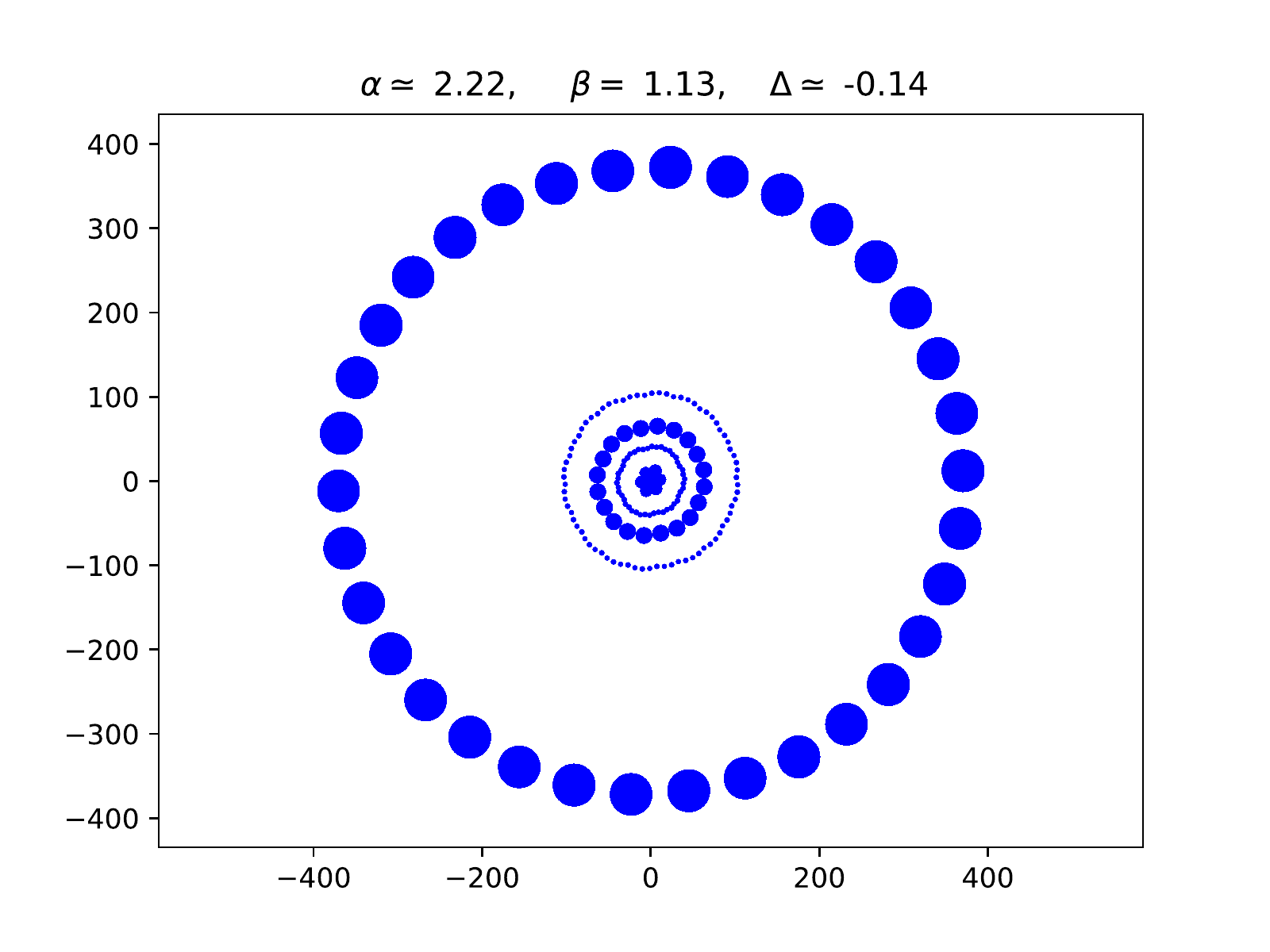}
\includegraphics[width=7.5cm]{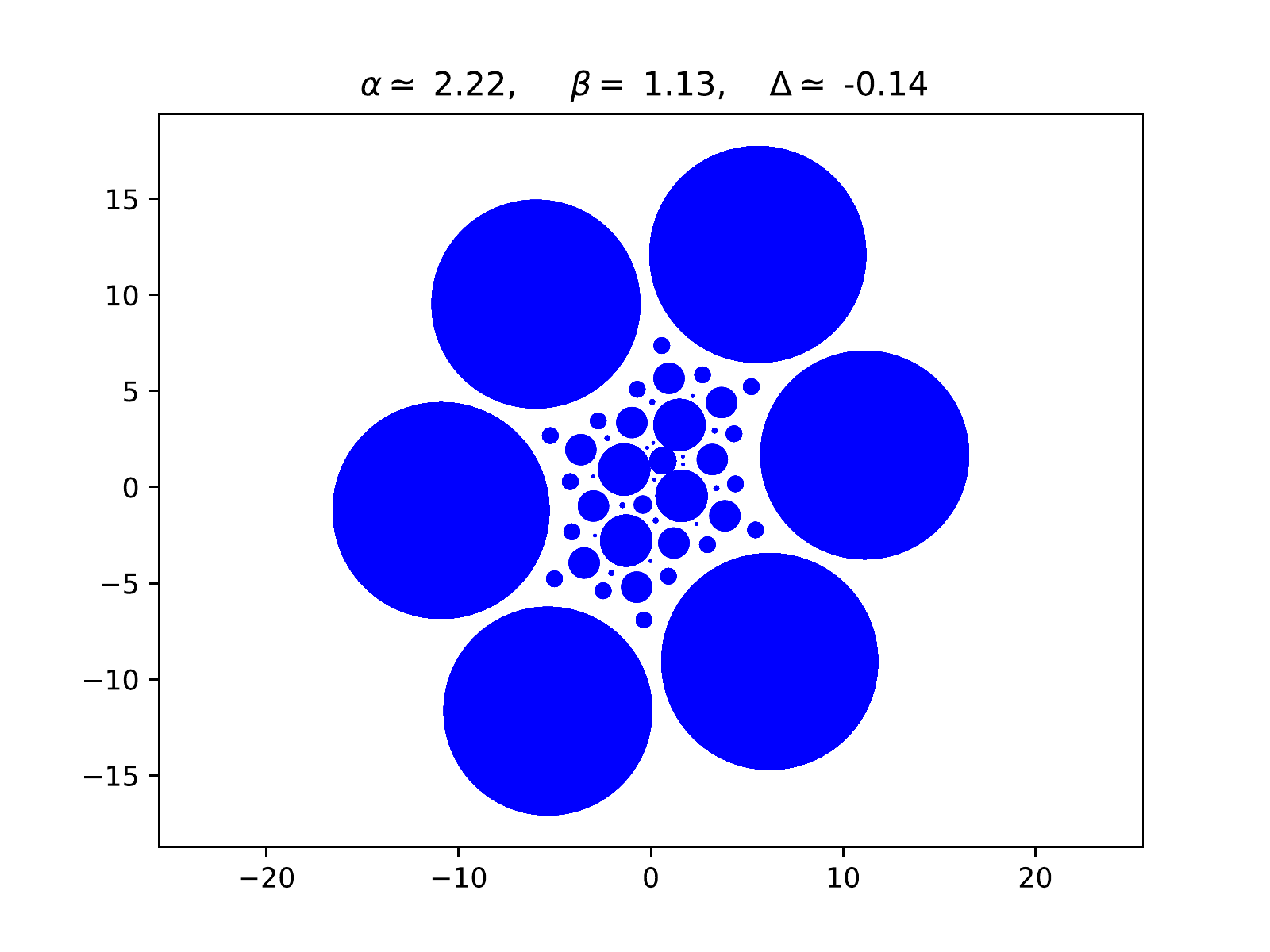}
\caption{Two pictures of the attractive set of the piecewise isometry of Example \ref{exemple_cas_irrationnel}.}\label{c-2}
\end{center}
\end{figure}

\section{Size of the limit set for the rational angle} \label{se:rational}

Now we prove the same type of result for a piecewise rotation of rational angle.

\subsection{Notations and statement}

Consider \textbf{a  non-bijective map} $T = T(\alpha,C_0,C_1,\mathbf D, \gamma)\in \mathcal T$. 
In all of this section we 
will use the following notation:  
$\alpha = \frac{p}{q} 2\pi$ with $q>2$ an even number and $\gcd(p,q)=1$. We will prove:

\begin{theoreme} \label{th:casrat}
Let $T\in\mathcal{T}$ a non-bijective map 
with a rational angle $\alpha$ such that $\alpha = \frac{p}{q} 2\pi$ for some even number $q>2$ 
and $\gcd(p,q)=1$. Then we obtain:

 $$M(T)\leq  q\norme{T}  \left(\frac{1}{ 2\tan(|\beta-\alpha/2 -\gamma +\pi/2 |)} + \frac{1}{2\tan(\pi/q) } +1 \right).$$

\end{theoreme}

\begin{remark}\label{rem:conj-pi/2}
Up to a conjugacy by $z\mapsto e^{-i(\gamma+\pi/2)}$ we can always assume
that 
$\gamma = \pi/2$ and that the discontinuity line is $\mathbf D = i\mathbb R$. We will assume 
this 
in the proof and 
thus 
we will prove:
$$M(T)\leq  q\norme{T}  \left(\frac{1}{ 2\tan(|\beta-\alpha/2|)} + \frac{1}{2\tan(\pi/q) } +1 \right).$$
\end{remark}

We finish this paragraph with another remark.
\begin{remark} \label{rem:alpha_pos}
If we conjugue by $z\mapsto \bar z$ (see Remark \ref{conjugaison}), then we can assume $\alpha\in]0,\pi[$.
This 
allows us to 
assume 
$\sin(\alpha/2)>0$.
\end{remark}

\subsection{Vectors of translations}

First of all, we consider the following partition of the complement of $B(0,q\norme{T})$, see Figure \ref{fig:substrip}:

\begin{itemize}
\item
For $k=0,\ldots,q-1$, we define the cone $\mathbf C_k = \left\{z\mid\frac{\pi}{2}+2\pi \frac{k}{q}<\arg(z)<\frac{\pi}{2}+2\pi\frac{(k+1)}{q}\right\}$.
\item
For $k=0,\ldots,q-1$, let us consider $\mathbf E_k$ which is the translate of $\mathbf C_k$ such that its vertex 
is the point
\[ O_k = \frac{q\norme{T}}{\sin(\pi/q)}\exp\left( i \left( \frac{\pi}{q} + 2\pi\frac{k}{q} + \frac{\pi}{2} \right)\right).\]
\item
For $k=0,\ldots,q-1$, let us denote $\mathbf G_k$ the strip which is the intersection of the set of points of modulus bigger than $q\norme{T}$ and of the closed $q\norme{T}$-neighborhood of $\mathbf C_{k-1} \cap \mathbf C_{k}$ where $\mathbf  C_{-1}= \mathbf C_{q-1}$.
\end{itemize}

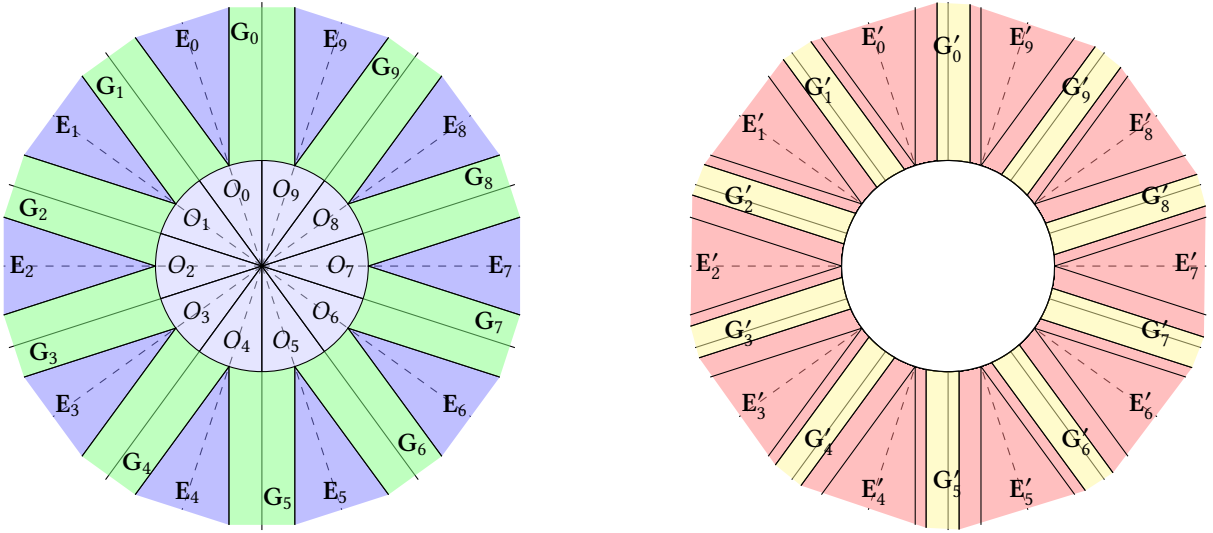
\begin{figure} 
\begin{center}
    \def \nombreq{5} 
    \pgfmathsetmacro{\nombreqb}{2*\nombreq-1} 
    \def \angle{180/\nombreq}
    	\pgfmathsetmacro{\sa}{sin(\angle)} 
    	\pgfmathsetmacro{\ca}{cos(\angle)} 
    	\pgfmathsetmacro{\ta}{tan(\angle)}
    	\pgfmathsetmacro{\smoi}{sin(\angle/2)} 
    	\pgfmathsetmacro{\cmoi}{cos(\angle/2)}
    	\pgfmathsetmacro{\tmoi}{tan(\angle/2)}		
    \def \longeur{3}
    \def \tr{2}
    \def \dep{1}
    \def \espace{0.5}
    \begin{tikzpicture}[scale=0.7]
    \foreach \x in {0,1,...,\nombreqb}{
    \begin{scope}[rotate=90+\x*\angle]
    	\draw[dashed] (0,0)--(\longeur*\ca/2+\longeur/2+\tr*\cmoi,\longeur*\sa/2+\tr*\smoi);
    	\draw (0,0)--(\longeur+\tr,0);
    	\draw[fill = blue!50,fill opacity= 0.5] 
    		(\longeur*\ca+\tr*\cmoi,\longeur*\sa+\tr*\smoi)--(\tr*\cmoi,\tr*\smoi)--(\longeur+\tr*\cmoi,\tr*\smoi);
    	\node at (\longeur*\cmoi+\tr*\cmoi-\espace*\cmoi,\longeur*\smoi+\tr*\smoi-\espace*\smoi) 
    		{\small $\mathbf E_{\x}$}; 
    \end{scope}
    }
    \foreach \x in {0,...,\nombreqb}{
    \begin{scope}[rotate=90+\x*\angle]
    	\clip (0,0) circle (\tr);	
    	\draw[fill = blue!20, fill opacity= 0.5] (\longeur*\ca,\longeur*\sa)--(0,0)--(\longeur,0);
    	\node at (\tr*\cmoi-\espace*\cmoi,\tr*\smoi-\espace*\smoi) {\small $O_{\x}$};
    \end{scope}
    }
    \foreach \x in {0,...,\nombreqb}{
    \begin{scope}[rotate=90+\x*\angle]
    	\draw[fill = green!50, fill opacity= 0.5]  
    		(\longeur+\tr*\cmoi,-\tr*\smoi) -- (\tr*\cmoi,-\tr*\smoi) arc (-\angle/2:\angle/2:\tr cm) -- (\longeur+\tr*\cmoi,\tr*\smoi);
    	\node at (\longeur+\tr*\cmoi-\espace,\tr*\smoi/2) {\small $\mathbf  G_{\x}$}; 
    \end{scope}
    }
    \end{tikzpicture}
\hfill
    \pgfmathsetmacro{\aa}{\tr*\smoi/3} 
    \pgfmathsetmacro{\bb}{2*\tr*\smoi/3} 
    	\pgfmathsetmacro{\aaangle}{asin(\aa/\tr)} 
    	\pgfmathsetmacro{\aac}{cos(\aaangle)} 
    	\pgfmathsetmacro{\aas}{sin(\aaangle)} 
    	\pgfmathsetmacro{\bbangle}{\angle-asin(\bb/\tr)} 
    	\pgfmathsetmacro{\bbc}{cos(\bbangle)} 
    	\pgfmathsetmacro{\bbs}{sin(\bbangle)} 
    	\pgfmathsetmacro{\bbbc}{cos(\bbangle-\angle)} 
    	\pgfmathsetmacro{\bbbs}{sin(\bbangle-\angle)} 
    \begin{tikzpicture}[scale=0.7]
    \foreach \x in {0,1,...,\nombreqb}{
    \begin{scope}[rotate=90+\x*\angle]
    	\draw[dashed] (0,0)--(\longeur*\ca/2+\longeur/2+\tr*\cmoi,\longeur*\sa/2+\tr*\smoi);
    	\draw (0,0)--(\longeur+\tr,0);
    	\draw[fill = red!50,fill opacity= 0.5] 
    		(\longeur+\tr*\aac,\tr*\aas)
    		-- (\tr*\aac,\tr*\aas) arc (\aaangle:\bbangle:\tr cm) 
    		-- (\longeur*\ca+\tr*\bbc,\longeur*\sa+\tr*\bbs);
    	\draw
    		(\longeur*\ca+\tr*\cmoi,\longeur*\sa+\tr*\smoi)--(\tr*\cmoi,\tr*\smoi)
    		--(\longeur+\tr*\cmoi,\tr*\smoi);
    	\node at (\longeur*\cmoi+\tr*\cmoi-\espace*\cmoi,\longeur*\smoi+\tr*\smoi-\espace*\smoi) 
    		{\small $\mathbf  E'_{\x}$}; \end{scope}
    }
    \foreach \x in {0,...,\nombreqb}{
    \begin{scope}[rotate=90+\x*\angle]
    	\draw[fill = white, fill opacity= 1] (0,0) circle (\tr);
    \end{scope}
    }
    \foreach \x in {0,...,\nombreqb}{\begin{scope}[rotate=90+\x*\angle]
    	\draw[fill = yellow!50, fill opacity= 0.5]  
    		(\longeur+\tr*\aac,\tr*\aas) -- (\tr*\aac,\tr*\aas) arc (\aaangle:\bbangle-\angle:\tr cm) 
    		-- (\longeur+\tr*\bbbc,+\tr*\bbbs);
    	\node at (\longeur+\tr*\ca-\espace,0) {\small $\mathbf  G'_{\x}$}; 
    	\end{scope}}
    \end{tikzpicture}
\end{center}
\caption{
	The different zones $\mathbf  E_k$, $\mathbf  G_k$,
	$\mathbf  E'_k$ and $\mathbf  G_k'$
	(cf Lemma \ref{lemme-faux}).
	}
\label{fig:substrip} 
\end{figure}

The map $T^q$ is clearly a piecewise translation. In the next lemma we compute the vectors of translation which define this map. We identify complex numbers with coordinates of points, in order to simplify.

\begin{lemma}\label{lemme-faux}
Let us denote:
$v=2\frac{\sin(\alpha/2)}{\sin(\pi/q)}|C_1-C_0|$ and $w=2\frac{\sin(\alpha/2)}{\tan(\pi/q)}|C_1-C_0|.$
For $k=0,\dots, q-1$, let us consider the vectors: 
\[
	v_k=ve^{i\left(\beta-\frac{\alpha}{2}+k\frac{2\pi}{q}+\frac{\pi}{q}\right)},
	\quad \text{and} \quad
	w_k=we^{i\left(\beta-\frac{\alpha}{2}+k\frac{2\pi}{q} \right)}.
\] 
Then the map $T^{q}$ is a piecewise translation with vectors $t_i$ with $i\in I$ for some finite set $I$. 
We have for each $k\in \{0,\dots,q-1\}$.

\begin{itemize}
\item The set  $\mathbf  E'_k$, where the vector of translation of $T^{q}$ is $v_k$, is a subcone of $\mathbf  C_k$ which contains $\mathbf  E_k$, see Figure \ref{fig:substrip}.
\item 
The restriction of $T^{q}$ to $\mathbf  G_{k}$ is a translation by some vector $t_k$ which 
is equal to
	$v_{k-1},w_k$ or $v_k$. 
	
More precisely, 
	\begin{itemize}
	\item consider the subset $\mathbf  G'_k$ of $\mathbf  G_k$ where $T^{q}$ is a translation by $w_k$. Then $\mathbf  G'_k$ is a substrip of $\mathbf   G_k$.
\item There exists a subset of $\mathbf  G_k$ inside $\mathbf  C_{k-1}$ where $t_k=v_{k-1}$.
\end{itemize}
\end{itemize}

\end{lemma}

Recall that we have assumed 
that $\alpha$ satisfies 
$\sin(\alpha/2)>0$ (cf Remark \ref{rem:alpha_pos}). Thus we have  $v = |v_k|$ and $w = |w_k|$.

\begin{remark}
A statement similar to Lemma \ref{lemme-faux} above appears already 
in [Bosh.Goet.03]. 
Their statement, however, was not quite correct: 
they claimed that the vectors of translation 
for 
$T^q$ restricted to $\mathbf G_k$ all have the same norm. Moreover the description on the sets $\mathbf G'_k, \mathbf  C'_k$ was not in their paper.
\end{remark}

\begin{proof}[Proof of Lemma \ref{lemme-faux}]
Let us fix $k=0,\dots, q-1$.
We consider 
several 
different cases. 
\begin{itemize}
\item 
Let $z\in \mathbf   E_k$. 
We have $e^{i\ell \alpha}O_k=O_{k+p \ell }$ by definition of $O_k$ for all $\ell\in \mathbb N$. 
This gives us $e^{i\ell \alpha}z\in \mathbf   E_{k+p\ell }$. 
Then by Lemma \ref{equ-rot} we obtain for $\ell  \leq q-1$,  $d(T^\ell (z), \mathbf   E_{k+p\ell })<q\norme{T}$. 
We deduce 
$$
	T^\ell (z)\in \mathbf  C_{k+p\ell }\quad  0\leq \ell \leq q-1.
$$
Thus the first $q-1$ iterates of $z$ under 
the action of 
$T$ meet all the cones $\mathbf  C_i$ and in particular no two subsequent iterates lie in the same cone.

Thus the restriction of $T^{q}$ to $\mathbf   E_k$ is an honest translation and not a piecewise translation. 
Now we will compute the vector of translation of $T^{q}$. 
We use the notation
$$
	\eta(i)=\begin{cases} 0\quad i\leq q/2,\\ 1\quad i>q/2.\end{cases}
$$
We compute the vector $T^{q}(z)$ for $z\in \mathbf   E_k$ by the following formula, see Equation \eqref{eq:itere:rot}
\[
	T^{q}(z)=z+(1-e^{i\alpha})\sum_{n=0}^{q-1}C_{\eta(k+np)} e^{i\alpha(q-n-1)}
	=z+(1-e^{i\alpha})\sum_{n=0}^{q-1}C_{\eta(k+np)} e^{i\alpha(-n-1)}.
\]
Modulo 
$q$ the map $n\mapsto k+np$ is a bijection since $\gcd(p,q)=1$
\[
T^{q}(z)
	=z+(1-e^{i\alpha})\sum_{m=0}^{q-1}C_{\eta(m)}e^{i\alpha(-1-\frac{m-k}{p})}
	=z+(1-e^{i\alpha})\sum_{m=0}^{q-1}C_{\eta(m)}e^{2i\pi\frac{-p-m+k}{q}}.
\]
Thus we obtain :
\begin{equation}\label{equbase}
T^{q}(z)=z+(1-e^{i\alpha})C_0\sum_{m=0}^{q/2-1}e^{2i\pi\frac{-p-m+k}{q}}
	+(1-e^{i\alpha})C_1\sum_{m=q/2}^{q-1}e^{2i\pi\frac{-p-m+k}{q}}
\end{equation}
 
By means of well known trigonometric formula 
we obtain
\begin{align*}
T^{q}(z)
	&=z+(1-e^{i\alpha})e^{2i\pi\frac{-p+k}{q}} 
		\left(C_0\sum_{m=0}^{q/2-1}e^{-2i\pi\frac{m}{q}}
			+C_1\sum_{m=q/2}^{q-1}e^{-2i\pi\frac{m}{q}}\right)\\
	&=z+(1-e^{i\alpha})e^{2i\pi\frac{k-p}{q}}
		(C_0-C_1)\frac{2}{1-e^{-2i\pi/q}}\\
	&=z+2\frac{\sin(\alpha/2)}{\sin(\pi/q)}|C_1-C_0|
		\exp\left( i \left( \beta+2\pi\frac{k-p}{q}+\frac{\alpha}{2}+\frac{\pi}{q} \right) \right)\\
	&=z+2\frac{\sin(\alpha/2)}{\sin(\pi/q)} |C_1-C_0|
		e^{i\left(\beta-\frac{\alpha}{2}+\frac{\pi}{q}+2\pi\frac{k}{q}\right)}\\
	&=z+v_k.
\end{align*}

\item 
Now consider $z\in \mathbf  G_k$. 
We will denote 
by 
$R$ the rotation 
with 
center $0$ and angle $\alpha$. 
If there exists some $n\geq 0$ such that 
$e^{in\alpha} z$ and $T^n(z)$ 
are in the two different half planes 
defined by $\re(z)\geq 0$ and $ \re(z)\leq 0$, then this implies that the two points 
are both in $\mathbf  G_0$ or $\mathbf  G_s$. 
Indeed by Equation \eqref{equ-rot} they are at distance at most $q\norme{T}$ of each other.  
Thus there are several choices for the different codings
: either the codings of the orbits of $T^nz, R^nz$ for $z\in \mathbf  G_k$ differ 
in 
one position 
$j_1$, or else  
they differ at two positions $j_1, j_2$. 
Note that $j_1, j_2$ are defined by 
$k+j_1p=q/2 \mod q$ or $k+j_2p=0 \mod q$, according  
to the points in 
$\mathbf  G_0$ or $\mathbf  G_s$. 
In other words: 
$R^{j_1}(z)$ and  $T^{j_1}(z)$ are 
contained in different half spaces.

Now we want to do the same computation as in the case $z\in \mathbf  E_k$. We can resume the situation 
as indicated in the lines below: 
the first one concerns the coding of a point in $\mathbf  E_k$, 
and the next three lines correspond to one change in the coding if the change 
is at position $j_1$ or $j_2$, or else to two changes in the coding.  

$$\begin{cases}
u_0\dots u_{j_1}\dots u_{j_2}\dots\\
u_0\dots u_{j_1}\dots (1-u_{j_2})\dots\\
u_0\dots (1-u_{j_1})\dots u_{j_2}\dots\\
u_0\dots (1-u_{j_1})\dots (1-u_{j_2})\dots\\
\end{cases}$$

We consider the two possibilities, starting from 
Equation \eqref{equbase}.
\begin{enumerate}
\item 
Assume that there is only one change. 
For example if one $0$ becomes a $1$, we need to change in Equation \eqref{equbase} 
the $C_0\sum_0^{q/2-1}$ into $C_0\sum_{1}^{q/2-1}$ 
and $C_1\sum_{q/2}^{q-1}$ into $C_1\sum_{q/2}^{q}$. We obtain the following:

\begin{align*}
T^{q}(z)
	&=z+(1-e^{i\alpha})e^{2i\pi\frac{k-p}{q}}
		\left(C_0\sum_{m=1}^{q/2-1}e^{2i\pi \frac{- m}{q}}+C_1\sum_{m=q/2}^{q}e^{2i\pi \frac{- m}{q}}
		\right)\\
	&=z+(1-e^{i\alpha})e^{2i\pi \frac{k-p}{q}}
		\left(C_0\sum_{m=0}^{q/2-1}e^{2i\pi \frac{- m}{q}}-C_0
		+C_1\sum_{m=s}^{2s-1}e^{2i\pi \frac{- m}{q}}+C_1\right)\\
	&=z+(1-e^{i\alpha})e^{2i\pi\frac{k-p}{q}}
		\left((C_0-C_1) \sum_{m=0}^{q/2-1}e^{2 \frac{- m}{q}} -(C_0-C_1)\right)\\
	&=z+(1-e^{i\alpha})e^{2i\pi\frac{k-p}{q}}\left((C_0-C_1)\frac{2}{1-e^{-2i\pi/q}}-(C_0-C_1)\right).\\
	&=z+(1-e^{i\alpha})e^{2i\pi\frac{k-p}{q}}(C_0-C_1)
		\left( \frac{1+e^{-2i\pi/q}}{1-e^{-2i\pi/q}} \right).\\
	&=z+(1-e^{i\alpha})e^{2i\pi \frac{k-p}{q}}(C_0-C_1)\frac{2\cos(\pi/q)e^{-i\pi/q}}{1-e^{-2i\pi/q}}.\\
	&=z+\cos\left({\pi}/{q}\right) e^{-i\frac{\pi}{q}}v_k\\
	&=z+w_k.
\end{align*}

\item The case where one $1$ becomes a $0$ is similar:
 we replace $C_0\sum_0^{q/2-1}$ by $C_0\sum_{1}^{q/2}$ 
and $C_1\sum_{q/2}^{q-1}$ by $C_1\sum_{q/2+1}^{q}$. We obtain the following:

\begin{align*}
T^{q}(z)
	&=z+v_k +(1-e^{i\alpha})C_0e^{2i\pi\frac{-p-q/2+k}{q}} -(1-e^{i\alpha})C_1e^{2i\pi\frac{-p-q/2+k}{q}}\\
	&=z+v_k -(1-e^{i\alpha})C_0e^{2i\pi\frac{-p+k}{q}} -(1-e^{i\alpha})C_1e^{2i\pi\frac{-p+k}{q}}\\
	&=z+\cos(\pi/q)e^{-i\frac{\pi}{q}}v_k\\
	&=z+w_k.
\end{align*}

\item In the last case (where we exchange one $0$ and one $1$), we need to change the $C_0\sum_0^{q/2-1}$ into $C_0\sum_{1}^{q/2}$ and $C_1\sum_{q/2}^{q-1}$ into $C_1\sum_{q/2+1}^{q}$.
\begin{align*}
T^{q}(z)
	&=z+(1-e^{i\alpha})e^{2i\pi\frac{k-p}{q}}
		\left(C_0\sum_{m=1}^{q/2}e^{2i\pi\frac{- m}{q}}+C_1\sum_{m=q/2+1}^{q}e^{2i\pi\frac{- m}{q}}\right)\\
	&=z+(1-e^{i\alpha})e^{2i\pi\frac{k-p}{q}}
		\left(C_0\sum_{m=0}^{q/2-1}e^{2i\pi \frac{- m}{q}}-2C_0+C_1\sum_{m=q/2}^{q-1}e^{2i\pi \frac{- m}{q}}+2C_1\right)\\
	&=z+(1-e^{i\alpha})e^{2i\pi\frac{k-p}{q}}
		\left(C_0\sum_{m=0}^{q/2-1}e^{2i\pi \frac{- m}{q}}
		+C_1\sum_{m=q/2}^{q-1}e^{2i\pi\frac{- m}{q}}+2(C_1-C_0)\right).\\
\end{align*}
The same computation shows that
\begin{align*}
T^{q}(z)
	&=z+(1-e^{i\alpha})e^{2i\pi\frac{k-p}{q}}(C_0-C_1)\left(\frac{2}{1-e^{-2i\pi/q}}-2\right)\\
	&=z+(1-e^{i\alpha})e^{2i\pi\frac{k-p}{q}}(C_0-C_1)\frac{2e^{-2i\pi/q}}{1-e^{-2i\pi/q}}\\
	&=z+v_ke^{-2i\pi/q}=z+v_{k-1}
\end{align*}
\end{enumerate}
\item
Let $z\in \mathbf  G_k$, consider $t>0$.
We claim that the codings of $z$ and $z'=z+te^{-i\pi/q}O_k$
have the same prefix of size $q+1$ ({\it i.e} $T^l(z)$ and $T^l(tz)$ are in the same half planes for
$l\in \{0,\ldots,q\}$). 
This proves that $\mathbf  G'_k$ is a substrip of $\mathbf {G}_k$, see Figure \ref{fig:substrip}, since for each point $z\in \mathbf  G'_k$ a half-line starting from $z$ will be inside $\mathbf  G'_k$.

\begin{flalign*}
T^{j_1}(z') - T^{j_1}(z) 
	& =  e^{2i\pi \frac{p}{q}j_1} z' + \text{Cst}
		- e^{2i\pi \frac{p}{q}j_1} z - \text{Cst}
	 =  e^{2i\pi \frac{p}{q}j_1} (z'-z)
	 =  e^{2i\pi \frac{p}{q}j_1} te^{-i\frac{\pi}{q}}O_j \\
	& =  e^{2i\pi \frac{p}{q}j_1} te^{-i\frac{\pi}{q}}
		\frac{q\norme{T} }{\sin(\pi/q)}e^{ i  \frac{\pi}{q} + 2i\pi\frac{j}{q} + i\frac{\pi}{2}} 
	 = \frac{q\norme{T} t }{\sin(\pi/q)}  i  e^{2i\pi \frac{p}{q}j_1}  e^{ 2i\pi\frac{j}{q}}\\
	&= \frac{q\norme{T} t }{\sin(\pi/q)}  i  e^{2i\pi \frac{pj_1+j}{q}}
	 = \frac{q\norme{T} t }{\sin(\pi/q)}  i  e^{2i\pi \frac{1}{2}}
	 = \frac{q\norme{T} t (-1)^k}{\sin(\pi/q)}  i  \in i\mathbb R\\
	&\implies \re(T^{j_1}(z))= \re(T^{j_1}(z')).
\end{flalign*}

Since by now all cases are treated, the claim is proved.
\end{itemize}
\end{proof}

\begin{notations}
Let us denote $a_k$ (or $b_k$ respectively) 
the width of the half strip $\mathbf   G_k'\cap \mathbf  C_k$ (or $\mathbf   G_k'\cap \mathbf  C_{k-1}$ respectively), 
see Figure \ref{fig:vecttran} and Figure \ref{fig:preuveconvpoly}. Note that we have $0\leq a_k,b_k\leq q\norme{T}$.
\end{notations}

\begin{figure} 
\begin{center}
    \def \nombreq{5}  \pgfmathsetmacro{\nombreqb}{2*\nombreq-1} 
    \def \angle{180/\nombreq}
    	\pgfmathsetmacro{\sa}{sin(\angle)} 
    	\pgfmathsetmacro{\ca}{cos(\angle)} 
    	\pgfmathsetmacro{\ta}{tan(\angle)}
    	\pgfmathsetmacro{\smoi}{sin(\angle/2)} 
    	\pgfmathsetmacro{\cmoi}{cos(\angle/2)}
    	\pgfmathsetmacro{\tmoi}{tan(\angle/2)}	
    \def \longeur{3}
    \def \tr{2}
    \def \dep{1}
    \def \espace{0.5}
    \pgfmathsetmacro{\aa}{\tr*\smoi/3} 
    \pgfmathsetmacro{\bb}{2*\tr*\smoi/3} 
    	\pgfmathsetmacro{\aaangle}{asin(\aa/\tr)} 
    	\pgfmathsetmacro{\aac}{cos(\aaangle)} 
    	\pgfmathsetmacro{\aas}{sin(\aaangle)} 
    	\pgfmathsetmacro{\bbangle}{\angle-asin(\bb/\tr)} 
    	\pgfmathsetmacro{\bbc}{cos(\bbangle)} 
    	\pgfmathsetmacro{\bbs}{sin(\bbangle)} 
    	\pgfmathsetmacro{\bbbc}{cos(\bbangle-\angle)} 
    	\pgfmathsetmacro{\bbbs}{sin(\bbangle-\angle)} 
    \def \llongueur{4}
    	\pgfmathsetmacro{\vecteurs}{sin(\angle/2+5)} 
    	\pgfmathsetmacro{\vecteurc}{cos(\angle/2+5)} 
    	\def \vecteurl{0.8}
    	\pgfmathsetmacro{\vecteurts}{sin(\angle/2+90)} 
    	\pgfmathsetmacro{\vecteurtc}{cos(\angle/2+90)} 
    \begin{tikzpicture}[scale=1]
    \draw[dashed] (0,0)--(\longeur*\ca/2+\longeur/2+\tr*\cmoi,\longeur*\sa/2+\tr*\smoi);
    \draw (0,0)--(\longeur+\tr,0);
    \draw (0,0)--(\longeur*\ca+\tr*\ca,\longeur*\sa+\tr*\sa);
    \draw[fill = red!50,fill opacity= 0.5] 
    		(\longeur+\tr*\aac,\tr*\aas)
    		-- (\tr*\aac,\tr*\aas) arc (\aaangle:\bbangle:\tr cm) 
    		-- (\longeur*\ca+\tr*\bbc,\longeur*\sa+\tr*\bbs);
    \draw[thick,->] (\vecteurc*\llongueur,\vecteurs*\llongueur)--(\vecteurc*\llongueur+\vecteurl,\vecteurs*\llongueur-\vecteurl)
    	node[right] {$v_k$};
    \begin{scope}[shift={(\vecteurc*\llongueur,\vecteurs*\llongueur)}]
    	\begin{scope}[rotate=-90+\angle/2]
    	\draw[ultra thick,green,->] (1.41*\vecteurl,0) arc (0:45-\angle/2:1.41*\vecteurl) ;
    	\end{scope}
    \end{scope}
    \draw[thick,->] (0,0) -- (\cmoi,\smoi) ;
    	\node[below] at (\cmoi,\smoi) {$e^{i \left( \frac{\pi}{q} + 2\pi\frac{k}{q} + \frac{\pi}{2} \right)}$};
    \draw[dashed] (\vecteurc*\llongueur-2*\vecteurtc,\vecteurs*\llongueur-2*\vecteurts)
    			--(\vecteurc*\llongueur+2*\vecteurtc,\vecteurs*\llongueur+2*\vecteurts);
\draw[decorate,decoration={brace,raise=0.1cm}]
	(3,-0.5) -- (0,-0.5)
	node[below=0.2cm,pos=0.5] {$x$};
\begin{scope}[rotate= \angle]
\draw[decorate,decoration={brace,raise=0.1cm}]
	(0,0) -- (3,0) 
	node[above=0.2cm,pos=0.5] {$x$};
\end{scope}
\draw[decorate,decoration={brace,raise=0.1cm}]
	(3+\tr*\aac,\tr*\aas) -- (3+\tr*\aac,0) 
	node[right=0.2cm,pos=0.5] {$a_k$};
\begin{scope}[rotate=\angle]
	\draw[decorate,decoration={brace,raise=0.1cm}]
		(3+\tr*\aac,0) -- (3+\tr*\aac,\tr*\bbbs) 
		node[right=0.2cm,pos=0.5] {$b_{k+1}$};
\end{scope}

\draw[ultra thick,blue] 
	(3.5,0)--(3.5,\tr*\aas)--(3.5*\ca-\tr*\bbbs*\sa, 3.5*\sa  +\tr*\bbbs*\ca)--(3.5*\ca,3.5*\sa);
\node[thick,blue] at (3.5,\tr*\aas) {$\bullet$}; 
	\node[thick,blue,left ] at (3.5,\tr*\aas) {$A_k$}; 
		\node[thick,blue,below] at (3.5,0) {$A'_k$};
\node[thick,blue] at (3.5*\ca-\tr*\bbbs*\sa, 3.5*\sa  +\tr*\bbbs*\ca){$\bullet$}; 
	\node[thick,blue, left ] at  (3.5*\ca-\tr*\bbbs*\sa, 3.5*\sa  +\tr*\bbbs*\ca) {$B_k$}; 
	\node[thick,blue, above left ] at  (3.5*\ca,3.5*\sa) {$B'_k$};
\draw[thick,blue] 
	(3.5+0.1,0)--(3.5+0.1,0.1)--(3.5,0.1);
\begin{scope}[rotate=\angle] 
	\draw[thick,blue]  (3.5+0.1,0)--(3.5+0.1,-0.1)--(3.5,-0.1);
\end{scope}
    \end{tikzpicture}
\hfill
    \def \llongueur{2}
    \def \vecteurl{0.6}
    \begin{tikzpicture}[scale=2]
    \draw[fill = yellow!50, fill opacity= 0.5]  
    		(\longeur,\tr*\aas)
    		-- (\tr*\aac-1,\tr*\aas) arc (\aaangle:\bbangle-\angle:\tr cm) 
    		-- (\longeur,\tr*\bbbs);
    \draw[dashed] (\tr-0.2,0)--(\longeur,0);
    \draw[dashed] (\llongueur,\tr*\aas/2-\vecteurl*1.41)--(\llongueur,\tr*\aas/2);
    \draw[thick,->] (\llongueur,\tr*\aas/2)--(\llongueur+\vecteurl,\tr*\aas/2-\vecteurl) node[right] {$w_k$};
    \begin{scope}[rotate=-90,shift={(-\tr*\aas/2,\llongueur)}]
    \draw[ultra thick,green,->] (1.41*\vecteurl,0) arc (0:45:1.41*\vecteurl);
    \end{scope}
    \draw[thick,->] (0,0) -- (0.5,0) ;
    \draw[ultra thick, blue] (\llongueur,\tr*\aas) -- (\llongueur,\tr*\bbbs) ;
        	\node[above left] at (\llongueur,\tr*\aas) {\textcolor{blue}{$A_k$}}; \node at (\llongueur,\tr*\aas) {\textcolor{blue}{$\bullet$}};
        	\node[left] at (\llongueur,0) {\textcolor{blue}{$A'_k=B'_{k-1}$}};  	\node at (\llongueur,0) {\textcolor{blue}{$\bullet$}};
        	\node[below left] at (\llongueur,\tr*\bbbs) {\textcolor{blue}{$B_{k-1}$}}; \node at (\llongueur,\tr*\bbbs) {\textcolor{blue}{$\bullet$}};
        	\node[below] at (0.5,0) {$e^{i \left( 2\pi\frac{k}{q} + \frac{\pi}{2} \right)}$};
    	\node[] at (0,0) {$\bullet$}; \node[below] at (0,0) {$0$};
    \draw[decorate,decoration={brace,raise=0.01cm}]
    	(\longeur,\tr*\aas) -- (\longeur,0) 
    	node[right=0.1cm,pos=0.5] {$a_k$};
    \draw[decorate,decoration={brace,raise=0.01cm}]
    	(\longeur,0) -- (\longeur,\tr*\bbbs) 
    	node[right=0.1cm,pos=0.5] {$b_k$};
\draw[decorate,decoration={brace,raise=0.1cm}]
	(\llongueur,-0.8) -- (0,-0.8)
	node[below=0.2cm,pos=0.5] {$x$};
    \end{tikzpicture}
\end{center}
\caption{Translation vectors $v_k, w_k$ in $\mathbf  E'_k$ (red part) and $\mathbf  G_k'$ (yellow)
(cf Lemma \ref{lemme-faux}).
}
\label{fig:vecttran}
\end{figure}
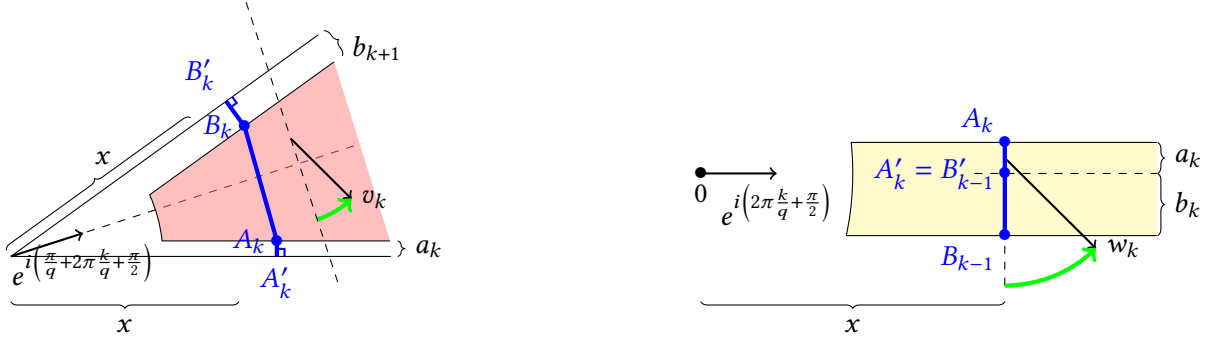

\subsection{Construction of polygons}

\subsubsection{
The case $\Delta<0$}

In this part we assume $\Delta<0$. By Definition of $\Delta$ and Remark \ref{rem:conj-pi/2}, we deduce that $\beta-\alpha/2$ belongs to $(0,\pi/2)$. The goal is to construct for each positive number $x$, big enough, 
some polygon $\mathbf P_x$. 
These polygons 
will have the following properties: 
\begin{itemize}
\item the disjoint union of all $\mathbf P_x$ cover the plane outside the ball $B(0,q\norme{T})$, 
\item the map $x\mapsto P_x$ is 
increasing, with respect to the partial order on subsets of the plane given by the inclusion, 
\item if $z$ 
lies outside of 
outside $\mathbf P_x$, then the piecewise translation $T^q$ maps $z$ to a point outside $\mathbf P_{x'}$ for some $x'>x$, 
see Proposition \ref{prop:polyTq}.
\end{itemize} This will be the key point to finish the proof of Theorem \ref{th:casrat}.

\begin{definition}
Consider $x>q\norme{T}$, then we define a polygon $\mathbf P_x$ with vertices $A_0,B_0,...,A_{q-1},B_{q-1}$ 
where (see Figure \ref{fig:vecttran}): 
\begin{itemize}
\item 
the point $A_k$ is defined in $\mathbf  E'_{k} \cap \mathbf  G'_{k}$ as the point 
such that the orthogonal projection $A'_k$ of $A_k$ on $[0,ie^{2i\pi k/q})$ has modulus $x$. 
\item
Similarly we define $B_k\in \mathbf  E'_{k} \cap \mathbf  G'_{k+1}$ such that the orthogonal projection $B'_k$ of $B_k$ on $[0,i e^{2i\pi (k+1)/q})$  has modulus $x$. 
\end{itemize}
Moreover we define
$\delta_k(x)$ as the angle (taken in $[0,\pi/2]$) between the lines $(A_kB_k)$ and $(A'_kB'_k)$.

\end{definition}

\begin{lemma}\label{le:modulepolygone}
For all $x>q\norme{T}$ and for all $z$ on the boundary of $\mathbf P_x$  we obtain $$x\cos(\pi/q) \leq |z|\leq \displaystyle\sqrt{x^2+q^2\norme{T}^2}.$$ 
\end{lemma}

The proof of this lemma is left to 
the 
reader with the help of Figure \ref{fig:vecttran}. 

\begin{lemma} \label{prop:convpoly}
Let us consider
$ \bar x = q\norme{T} \left(\frac{1}{ 2\tan(\beta-\alpha/2)} + \frac{1}{2\tan(\pi/q) } \right)$.
Then 
one has:
\begin{itemize}
\item the map $x\mapsto \delta_k(x)$ is decreasing
\item $\lim_{x\to+\infty} \delta_k(x)=0$ 
\item For $x> \bar x$, we have  $0\leq \delta_k(x) < \beta-\alpha/2$.  
\end{itemize}

\end{lemma}

\begin{proof}
We refer to Figure \ref{fig:preuveconvpoly} . 
There is a polygon with vertices $A'_k,A_k,B_k, B'_k$ and we define points $C_k, L_k$ such that:
\begin{itemize}
\item
If $a_k\geq b_{k+1}$, then $C_k$ belongs to $[A'_k,A_k]$ and $|A'_kC_k| = b_{k+1}$. The point $L_k$ is the orthogonal projection of $A_k$ on $[B_k,C_k]$.
\item
If $b_{k+1}>a_k$, $C_k$ belongs to $[B'k,Bk]$ and $|B'_kC_k| = a_k$. The point
$L_k$ is defined as the orthogonal projection of $A_k$ on $[A_k,C_k]$.
\end{itemize}
Let us denote $r_k = \max(a_k,b_{k+1}) - \min(a_k,b_{k+1})$ and $s_k = \min(a_k,b_{k+1})$.
Note 
that $0\leq r_k \leq q \norme{T}$.

\newpage
We present the proof in the case 
$a_k\geq b_{k+1}$; 
the other case is similar.
We have $\widehat{B_kC_kA_k}=\widehat{B_kA'_kA_k} = \pi/q$ and 
\[
\tan(\delta_k(x))  
	=  \tan(\widehat{C_kB_kA_k})  
	= \frac{|A_kL|}{|B_kL_k|}= \frac{|A_kL_k|}{|B_kC_k|-|C_kL_k|}
	=  \frac{r_k \sin(\pi/q)}{ |B_kC_k| -  r_k \cos(\pi/q) }	.
\]
We obtain 
from a classical geometry argument that 
$|B_kC_k| = \frac{x-s_k}{x} |B'_kA'_k| = 2(x-s_{k})\sin(\pi/q)  $. We deduce
\[
\tan(\delta_k(x))  
	=  \frac{r_k \sin(\pi/q)}{ 2(x-s_{k})\sin(\pi/q) -  r_k \cos(\pi/q) }
\implies
\delta_k(x)
	=  \arctan \left( \frac{r_k \sin(\pi/q)}{ 2(x-s_{k})\sin(\pi/q) -  r_k \cos(\pi/q) } \right).
\]
Thus we obtain that the map $x\mapsto \delta_k(x)$ is decreasing and tends to $0$ when $x$ approaches $+\infty$. 

The angle $\delta_k(x)$ is maximal if $r_k = q\norme{T}$ and $s_k=0$:
\[
\tan(\delta_k(x))  
	\leq  \frac{q\norme{T} \sin(\pi/q)}{ 2x\sin(\pi/q) -  q\norme{T} \cos(\pi/q) }.
\]
Thus we have :
\begin{flalign*}
\delta_k(x) < \beta-\alpha/2
	&\iff 
\tan(\delta_k(x) ) < \tan(\beta-\alpha/2)\\
	&\impliedby
\frac{q\norme{T} \sin(\pi/q)}{ 2x\sin(\pi/q) - q\norme{T} \cos(\pi/q) } < \tan(\beta-\alpha/2) \\
	& \impliedby
\frac{q\norme{T} \sin(\pi/q)}{\tan(\beta-\alpha/2)} <  2x\sin(\pi/q) - q\norme{T} \cos(\pi/q)  \\
	& \impliedby
\frac{q\norme{T}  \sin(\pi/q)}{ \tan(\beta-\alpha/2)} + q\norme{T} \cos(\pi/q)  < 2x\sin(\pi/q) \\
	&\impliedby 
\frac{q\norme{T}}{ 2\tan(\beta-\alpha/2)} + q\norme{T} \frac{1}{2\tan(\pi/q) } < x\\
	&\impliedby 	
\bar x < x.
\end{flalign*}
\end{proof}

\begin{remark}
We can also show $x \leq |A_k| =\sqrt{x^2+a_k^2} \leq \sqrt{x^2+q\norme{T}}$ and $x \leq |B_k| =\sqrt{x^2+b_{k}^2} \leq \sqrt{x^2+q\norme{T}}$.
Using 
Proposition \ref{prop:convpoly}, we deduce that $\frac 1x P_x$ 
converges in  
the Hausdorff topology to the regular polygon centered at $0$ with $q$ sides, and with $i$ 
as a 
vertex.
\end{remark}

\begin{figure} 
\begin{center}
\begin{tikzpicture}[scale=1]
\draw[decorate,decoration={brace,raise=0.5cm}]
	 (6.5,3) -- (8,0) 
	node[right=0.6cm,pos=0.4] {\small $a_k$};
\draw[decorate,decoration={brace,raise=0.1cm}]
	 (6.5,3) -- (7.5,1) 
	node[right=0.1cm,pos=0.5] {\small $r_k$};
\draw[decorate,decoration={brace,raise=0.1cm}]
	 (0,0)  -- (0.5,1)
	node[left=0.2cm,pos=0.5] {$s_k = b_{k+1}$};
\draw[decorate,decoration={brace,raise=0.2cm}]
	(8,0)--(0,0)
	node[below=0.4cm,pos=0.5] {$2x\sin(\pi/q)$};
\draw[thick,blue] 
	(0,0)--(0.5,1)--(6.5,3)--(8,0)--cycle;
	\draw[dashed,blue] (0,0)--(2, 4);
	\draw[dashed,blue] (8,0)--(6, 4);
	\draw[blue] (0.5,1)--(7.5, 1);	
\node[thick,blue, below right] at (8,0) {$A'_k$};
	\node[thick,blue] at (8,0) {$\bullet$};
\node[thick,blue, above right] at (6.5,3) {$A_k$};
	\node[blue] at (6.5,3) {$\bullet$};
\node[blue] at (6,4) {$\bullet$};
\node[thick,blue, below left] at (0,0) {$B'_k$};
	\node[thick,blue] at (0,0) {$\bullet$};
\node[thick,blue, above left] at (0.5,1) {$B_k$};
	\node[blue] at (0.5,1) {$\bullet$};
\node[blue] at (2,4) {$\bullet$};
\node[thick,blue, below] at (7.5,1) {$C_k$};
	\node[blue] at (7.5,1) {$\bullet$};
\draw[blue]  (6.5,3)--(6.5,1);
	\node[thick,blue, below] at (6.5,1)  {$L_k$};
\draw[-] (0.5,0) arc (0:63:0.5); 
	\node[right] at (0.5,0.2) {\small $\pi/q$};
\draw[-] (7.5,0) arc (180:117:0.5); 
	\node[left] at (7.5,0.2) {\small $\pi/q$};
\draw[-] (2,1) arc (0:18:1.5); 
	\node[right] at (2,1.2) {\small $\delta_k(x)$};
\end{tikzpicture}

\end{center}
\caption{The quantity $\delta_k(x)$ used in Lemma \ref{prop:convpoly}. }
\label{fig:preuveconvpoly}
\end{figure}
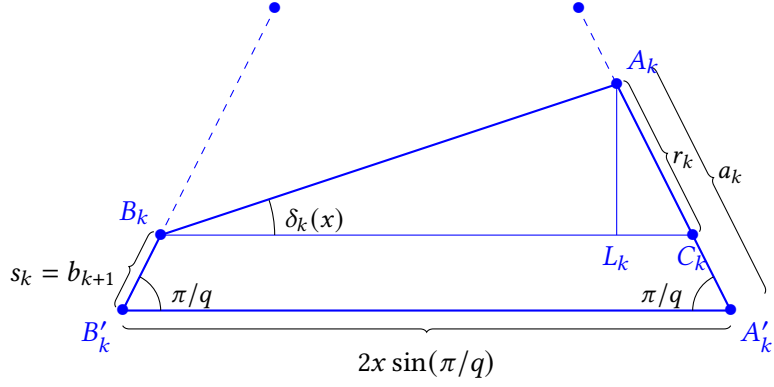

\begin{proposition} \label{prop:polyTq} 

For 
any $x > \bar x$ 
let us define $K_x = \min\big( v\sin(\beta-\alpha/2-\delta_k(x)) , w\sin(\beta-\alpha/2) \big)$, see Lemma \ref{lemme-faux} for notations. 
Then, if $z\in \mathbf P_{x}$, 
we have $T^{q}(z)\in \mathbf P_{x'}$ with $x'\geq x+K_x$. 
Moreover, 
the map $x\mapsto K_x$ is increasing.
\end{proposition}

\begin{proof}
First of all, 
note that by Lemma \ref{prop:convpoly} 
we have 
$\beta-\alpha/2-|\delta_x| > 0$.
From now on we will use the notation from Figure \ref{fig:preuveconvpoly}. 

We split the proof 
into 
four cases, due to Lemma \ref{lemme-faux}.
In each case we consider $z\in P_x$ and compute $x'$ such that $T^qz\in P_{x'}$.

\begin{itemize}
\item 
Assume 
$z\in \mathbf  E_k'$ and $T^q(z)\in \mathbf  E_k'$. 
The angle between $v_k$ and the blue segment is equal to  $\beta-\alpha/2-\delta_k(x)$. 
Indeed, 
$\beta-\alpha/2$ is the angle between $A'_kB'_k$ and $v_k$, see Figure \ref{fig:vecttran}. 
Thus 
we deduce
$x' \geq  x+v_k\sin(\beta-\alpha/2-\delta_k(x))$.

\item If $z\in \mathbf  E_k'$ and $T^q(z)\in \mathbf  G_{k}'$, then $x'-x$ is bigger than previous value 
(see green arrow in Figure \ref{fig:vecttran}).

\item 
Assume 
$z\in \mathbf  G_k'$ and $T^q(z)\in \mathbf  G_k'$

Then we obtain $x' = x+w\sin(\beta-\alpha/2)$.
\item 
Assume $z\in \mathbf  G_k'$ and $T^q(z)\in \mathbf  E_{k-1}'$. 
In this case we can fall back onto the previous case $x' \geq x+w\sin(\beta-\alpha/2)$.
\end{itemize}

Since $x \mapsto \delta_k(x)$ is decreasing, we deduce that $x \mapsto K_x$ is increasing.
\end{proof}

\subsubsection{The case $\Delta>0$}
The same construction 
as in the case $\Delta<0$ can be done with a familly of polygons $\mathbf Q_x$:

\begin{itemize}
\item
As in Lemma \ref{le:modulepolygone} we can prove:
\[ 
	\forall z\in \mathbb C, \quad
	z\in Q_x \implies x \cos(\pi/q) \leq |z|\leq \sqrt{x^2+q^2\norme{T}^2}
\]
\item
As in Proposition \ref{prop:polyTq},
for all $x> \bar x$ (with $\bar x$ defined 
as 
in Proposition \ref{prop:convpoly}), there exists $L_x>0$ such that for 
$z\in \mathbf Q_{x}$ we have $T^{q}(z)\in \mathbf P_{x'}$ with $x'\leq x-L_x$.
Moreover, 
the map $x\to L_x$ is increasing and 
tends to $0$ if $x\to \bar x$.
\end{itemize}

\subsection{Proof of Theorem \ref{th:casrat}}

We consider two cases:

\begin{itemize}
\item First case: $\Delta<0$.

Consider $z$ such that $|z|>\sqrt{{\bar x}^2+q^2\norme{T}^2}
$. 
By 
Lemma \ref{le:modulepolygone} we deduce $z\in \mathbf P_{x_0}
$, 
with 
\[
	x_0 \geq \sqrt{|z|^2-q^2\norme{T}^2}> \sqrt{{\bar x}^2}=\bar x.
\]
Thus we have $x_0> \bar x$ and by Proposition \ref{prop:polyTq}, for all $n\in \mathbb N$, we obtain
$T^{nq }z\in \mathbf P_{x_n}$ with 
\[
x_n\geq x_0+K_{x_0}+\cdots+K_{x_{n-1}} \geq x_0+nK_{x_0} 
	\text{ since $x \mapsto K_x$ is increasing}.
\]
Now by Lemma \ref{le:modulepolygone} we have
\[
	|T^{nq }z| \geq  x_n\cos(\pi/q) \geq (x_0+K_{x_0}n)\cos(\pi/q) \to+\infty.
\]
Thus we deduce $M\leq \sqrt{{\bar x}^2+q^2\norme{T}^2}$. We obtain finally
\begin{flalign*}
\sqrt{{\bar x}^2+q^2\norme{T}^2}
	&=
q\norme{T} \sqrt{ \left(\frac{1}{ 2\tan(|\beta-\alpha/2|)} + \frac{1}{2\tan(\pi/q) } \right)^2+1}\\
	&\leq
q\norme{T}  \left(\frac{1}{ 2\tan(|\beta-\alpha/2|) } + \frac{1}{2\tan(\pi/q) } +1 \right).
\end{flalign*}

\item Second case: $\Delta>0$.
Let us denote 
by 
$\conv(Q_x)$ the convex hull of $\mathbf Q_x$.

We fix 
$M>0$, and 
we 
deduce $B(0,M)\subset \conv(Q_{x_0})$ with $x_0 = M/\cos(\pi/q)$.

If $x_0\leq \bar x$ then the proof is finished. 

Otherwise, 
consider  a sequence $(x_n)_n$ of real numbers, defined by induction according to the following instructions:
If 
$x_{n-1}>\bar x$, 
then 
there exists $x_n>0$ such that $T^{qn}B(0,M)\subset \conv(Q_{x_n})$ and $x_n<x_{n-1}<\cdots<x_0$.

If there exists $n_1$ such that $x_{n_1}<\bar x$, then the proof is 
finished. 
Otherwise the sequence $(x_n)_n$ is a 
decreasing 
bounded sequence, thus it converges to 
some value 
$x_\infty$.
If $x_\infty>\bar x$, then $L_x$ has limit 
value 
$0$ as $x$ 
tends to $x_\infty$, 
and 
by a fact directly analogous to what has been shown in Proposition \ref{prop:polyTq}  we deduce 
$L_{x_\infty}=0$, 
which is a contradiction. We conclude $x_\infty=\bar x$.

The convex polygons $(\conv(Q_x))_x$ form a decreasing 
family (again with respect to the inclusion)
and 
from the definition of $\attr(T)$ 
we deduce
\[
\attr(T) \subset \bigcap _{n\geq 0}T^{qn}B(0,M)\subset \bigcap _{n\geq 0} \conv(x_n) = \conv(Q_{\bar x}).
\]
We conclude the proof with the 
observation that 
$
	\conv(Q_{\bar x})\subset B(0,\sqrt{{\bar x}^2+q^2\norme{T}^2}).
$
\end{itemize}

\begin{remark}
The above derived result can be improved slightly: indeed, it is possible to show
$\born(T)\subset\conv(P_{\bar x})$ and $\attr(T)\subset\conv(Q_{\bar x})$.
\end{remark}

\subsection{Rational example similar to Section \ref{exemple_cas_irrationnel}}

We consider the irrational angle of Section \ref{exemple_cas_irrationnel} and truncate it with the convergents of the continued fraction. In other 
words, we consider 
the following:
\[
\renewcommand{\arraystretch}{1.5}
\begin{array}{|c|c|c|c|c|c||c|c|c|c|}
\hline
\alpha & a & C_0 & C_1  & \gamma & \mathbf D &
	\beta & \Delta \\
\hline 
\frac{169}{478} 2\pi& \frac{p_7}{q_7} = \frac{169}{478} &-e^{1.14i}& 1.5e^{1.14i}& \frac{\pi}{2}& i\mathbb R &
	1.14 & -0.13112>\Delta >-0.13111  \\
\hline 
\end{array}
\]

Here we obtain $M \simeq 120\, 968$. 
For the irrational angle, we have obtained $M\simeq 571283$.

\bibliographystyle{abbrv}
\bibliography{biblio-attrBBK.bib}
\end{document}